\pdfoutput=1
\documentclass[11pt,a4paper]{amsart}
\usepackage{amssymb}
\usepackage{amsthm}
\usepackage{amsmath}
\usepackage{amsxtra}
\usepackage{latexsym}
\usepackage{mathrsfs}
\usepackage[all,cmtip]{xy}
\usepackage[all]{xy}
\usepackage{enumitem}
\usepackage{xcolor}
\usepackage{graphicx}
\usepackage{comment}
\usepackage{mathabx,epsfig}
\usepackage{stmaryrd}
\usepackage{comment}
\usepackage{pst-node}
\usepackage{mathtools}

\usepackage{tikz-cd} 

\usepackage{hyperref}

\newcommand{\RNum}[1]{\uppercase\expandafter{\romannumeral #1\relax}}

\newtheorem{thm}{Theorem}[section]
\newtheorem{lem}[thm]{Lemma}

\newtheorem{prop}[thm]{Proposition}

\theoremstyle{definition}
\newtheorem{defn}[thm]{Definition}
\newtheorem{rmk}[thm]{Remark}

\newtheorem{question}[thm]{Question}

\numberwithin{equation}{section}

\newcommand\be{\begin{equation}}
\newcommand\ba{\begin{eqnarray}}
\newcommand\ee{\end{equation}}
\newcommand\ea{\end{eqnarray}}
\newcommand\restr[2]{\ensuremath{\left.#1\right|_{#2}}}
\def\C{{\mathbb C}}
\def\Q{{\mathbb Q}}
\def\R{{\mathbb R}}
\def\Z{{\mathbb Z}}
\def\P{{\mathbb P}}

\def\N{{\mathbb N}}

\def\K{{\mathcal{K}}}

\DeclareMathOperator{\Aut}{Aut}

\DeclareMathOperator{\Frac}{Frac}

\DeclareMathOperator{\val}{val}

  {\end{list}}

\title[Common zeros of iterated rational functions]
{A finiteness result for common zeros of iterates of rational functions}
\thanks{The second author was supported in part by NSERC grant RGPIN-2022-02951.}

\author[Chatchai Noytaptim]{Chatchai Noytaptim}
\address{University of Waterloo \\
Department of Pure Mathematics \\
Waterloo, Ontario \\
Canada  N2L 3G1}
\email{cnoytaptim@uwaterloo.ca, chatchai.noytaptim@gmail.com}
\author{Xiao Zhong}
\address{University of Waterloo \\
Department of Pure Mathematics \\
Waterloo, Ontario \\
Canada  N2L 3G1}
\email{x48zhong@uwaterloo.ca}
\date{\today}
\subjclass[2020]{37P05, 37P30}
\keywords{arithmetic dynamics, iterated rational functions, arithmetic equidistribution, common zeros}
\begin{document}

\maketitle
\begin{abstract}
    Answering a question asked by Hsia and Tucker in their paper on the finiteness of greatest common divisors of iterates of polynomials, we prove that if $f, g \in \C(X)$ are compositionally independent rational functions and $c \in \C(X)$, then there are at most finitely many $\lambda\in\mathbb{C}$ with the property that there is an $n$ such that $f^n(\lambda) = g^n(\lambda) = c(\lambda)$, except for a few families of $f, g \in \Aut(\P^1_\C)$ which give counterexamples.
\end{abstract}
\section{Introduction}
\subsection{Historical background.} Bugeaud, Corvaja, and Zannier \cite{BCZ03} provided an upper bound on the greatest common divisor (GCD) of two integer sequences using tools from Diophantine approximation. More precisely, they showed that for any multiplicatively independent integers $a,b\geq 2$ and for any $\epsilon>0$, then
$$\mathrm{gcd}(a^n-1,b^n-1)<\mathrm{exp}(\epsilon n)$$ for sufficiently large $n\geq1.$\\
\indent In 2004, Ailon and Rudnick \cite{AR04} obtained the first instance of a GCD-type result over function field of characteristic $0$. In fact, they proved that for any nonconstant multiplicatively independent polynomials $a,b\in\mathbb{C}[x]$ there exists a polynomial $h\in \mathbb{C}[x]$ such that $\mathrm{gcd}(a^n-1,b^n-1)|h$ for all $n\geq 1.$ Ostafe \cite{Os16} pointed out that the result of Ailon-Rudnick holds generally for any pair $(m,n)\in\mathbb{N}\times \mathbb{N}$. Ostafe also provided an upper bound on the degree of $\mathrm{gcd}(a^m-1,b^n-1)$ in terms of degrees of $a$ and $b$.\\
\indent There are many generalizations and extensions of GCD-type results in various contexts, see for example \cite{CZ13}, \cite{GHT17}, \cite{GHT18}, \cite{Gr20}, \cite{GW20}, \cite{Hu20}, \cite{Le19}, \cite{Lu05}, \cite{Si04a}, \cite{Si04b}, \cite{Si05} and \cite{Za12}.\\
\indent Motivated by taking multiplicative powers of polynomials, one can study GCD-type results of iterates of rational functions. Let $f(x)\in\mathbb{C}(x)$ be a rational function with complex  coefficients. Here and what follows, $f^n$ denotes the $n$-fold composition of $f$ with itself.
\begin{defn} Two rational functions $f$ and $g$ are said to be compositionally independent if there exist positive integers $i_1,...,i_s,j_1,...,j_s,l_1,...,l_t,m_1,...,m_t$ such that
$$f^{i_1}\circ g^{j_1}\circ \cdots \circ f^{i_s}\circ g^{j_s}=f^{l_1}\circ g^{m_1}\circ \cdots \circ f^{l_t}\circ g^{m_t}$$
then $s=t$ and $i_k=l_k, j_k=m_k$ for $1\leq k\leq s.$
In other words, the semigroup generated by $f$ and $g$ under composition is isomorphic to the free semigroup of two generators.
\end{defn}

Hsia and Tucker \cite{HT17} first addressed the question posed by Ostafe \cite[$\S4.2$]{Os16} in dynamical setting. Under suitable compositional assumptions, they showed that for any polynomial maps $f, g$ of degree $\geq 2$  with complex coefficients and $c(x)\in\mathbb{C}[x]$ there are only finitely many irreducible factors of $$\mathrm{gcd}(f^m(x)-c(x), g^n(x)-c(x))$$ for some positive integers $m$ and $n$ (cf. \cite[Theorem 1]{HT17}). Interestingly, Theorem $2$ in \cite{HT17} can be viewed as a compositional analog result of Ailon-Rudnick  for linear polynomials. However, for nonlinear polynomials, their result holds conditionally on the assumption that $c(x)$ is not a complex constant which is in ramified cycle of both $f$ and $g$ (cf. \cite[Thoerem 4]{HT17}). 

\subsection{Statement of main results.} \label{sec: mainresults}
 In this article, we answer the following  question posed by Hsia and Tucker.   
 \begin{question} \cite[Question 18]{HT17} \label{qu: HTquestion}
Let $f,g$ be two nonconstant, compositionally independent rational functions with complex coefficients.   Let $c$ be any rational function with complex coefficients. Is it true that there must be at most finitely many $\lambda\in\mathbb{C}$ such that $f^n(\lambda)=g^n(\lambda)=c(\lambda)$ for some positive integer $n$?
\end{question}
It turns out that Question \ref{qu: HTquestion} is not true in general. This can be seen from the counter-examples provided in Remark (\ref{rmk: counter-1}),(\ref{rmk: counter-2}), (\ref{rmk: counter-3}), (\ref{rmk: counter-4}) and (\ref{rmk: counter-5}).\\

However, after excluding these families of rational functions $f$ and $g$, we obtain a finiteness result of common zeros of iterated rational functions in a similar spirit to Question \ref{qu: HTquestion}. More precisely, we state our main results as follows.
 \begin{thm}  \label{thm: mainlinear}  Let $f(X)$, $g(X)$ be automorphisms on $\P^1$ defined over $\C$ and $c(X)$ be a rational function defined over $\C$. If the semigroup generated by $f(X)$ and $g(X)$ under compositions is free, then there are only finitely many $\lambda \in \C$ such that 
    \begin{equation} \label{eq: targeteq1}
       f^n(\lambda) = g^n(\lambda) =  c(\lambda) 
    \end{equation}
    for some positive integer $n$ unless $f(X)$, $g(X)$ are simultaneously conjugated by an automorphism on $\P^1(\C)$ to one of the following:
    \begin{enumerate}
        \item $\alpha X + \beta$, $X/(\gamma X + \delta)$, with some $\alpha, \delta, \gamma, \beta \in \C^*$ such that one of $\alpha /\delta $ and $\alpha \delta$ is a root of unity;
        \item $\alpha X + \beta$, $\delta X + \gamma$,  with some $\alpha, \delta \in \C^*$ such that $\alpha$ and $\delta$ are not roots of unity, $\gamma$ and $\beta$ are not both $0$, and either $\alpha /\delta $ is a root of unity other than $1$ or one of $\alpha^2/\delta$ and $\delta^2/\alpha$ is a root of unity.
    \end{enumerate}
 \end{thm}
 When at most one of rational functions $f$ and $g$ is of degree one, we can relax some conditions on $(m,n)\in\mathbb{N}\times \mathbb{N}$. Also, the finiteness result on the number zeros of both equations $f^m(x)-c(x)=0$ and $g^n(x)-c(x)=0$ holds under natural assumptions. This is reminiscent of polynomial maps counterpart (cf. \cite[Theorem 1]{HT17}).
 \begin{thm} \label{thm : mainnonlinear} Let $f$ and $g$ be two compositionally independent rational functions with complex coefficients,  at least one of which has degree greater than one. Suppose that $c(x)\in \mathbb{C}(x)$ is not a compositional power of $f$ or $g$. Then there are only finitely many $\lambda\in\mathbb{C}$ such that
 $$f^m(\lambda)=g^n(\lambda)=c(\lambda)$$
for some positive integers $m$ and $n$.
 \end{thm}
 Notice that the same statement allowing distinct $n$ and $m$ for the case that $f$ and $g$ are both automorphisms on $\P^1(\C)$ doesn't hold. There is a counter-example in the beginning of Section $3$ of \cite{HT17}.
 \subsection{Proof strategy.}
  Hsia and Tucker \cite{HT17} combine various tools and techniques from Diophantine geometry, complex dynamics, and number theory. Our approach is akin to their method. There are technical difficulties to overcome for rational functions. We briefly explain our ideas of proving Theorem \ref{thm: mainlinear} and Theorem \ref{thm : mainnonlinear}.\\
 \indent For automorphisms $f$ and $g$ defined over $\overline{\mathbb{Q}}$, we heavily hinge on tools from Diophantine geometry (e.g., Roth's theorem, Siegel's theorem, and $S$-unit theorem). After suitable excluding $f,g$ (see $\S \ref{sec: mainresults}$), we show that $f$ and $g$ must share a common fixed point (see Lemma \ref{lem: reduce-to-sharingfixpoint}). Then we dynamically conjugate by moving such fixed point to $\infty$, we may assume that $f(x)=\alpha x+\beta$ and $g(x)=\delta x+\gamma$ with $\alpha, \beta, \gamma, \delta\in\overline{\mathbb{Q}}^*$. We break our computation into several cases depending on whether $\alpha$ and $\delta$ is a root of unity to show that if there are infinitely many solutions to 
 $$f^n(x)=g^n(x)=c(x)$$ where $c(x)$ is any rational function with complex coefficients, then $f$ and $g$ are not compositionally independent. In other words, the semigroup generated by $f$ and $g$ under composition is not free (see Theorem \ref{thm: mainaut}). Using specialization argument, we obtain the desired result over $\mathbb{C}$. \\
\indent When exactly one  of rational functions $f, g\in\mathbb{C}(x)$ has degree greater than one, the proof follows verbatim in the polynomial maps setting as a consequence of Northcott's property explained in \cite[Proposition 9]{HT17}. \\
\indent In the case both rational functions $f,g\in\mathbb{C}(x)$ are  of degree greater than $1$  and $c(x)\in\mathbb{C}(x)$. Under suitable compositional assumptions, we show that the infinite sequence of solutions to $$f^{m_i}(x_i)=g^{n_i}(x_i)=c(x_i)$$ as $i\rightarrow+\infty$ (so $m,n\rightarrow+\infty$) form a sequence of small dynamical height (Lemma \ref{smallpoints}).
We, then, apply the arithmetic equidistribution theorem of small points (in reminiscence of \cite{BD11}, \cite{BD13}, \cite{GHT13}, \cite{GHT15}, and \cite{PST12}) to deduce that  dynamical height attached to $f$ and $g$ must be identical (Proposition \ref{prop: shareheight}) and hence $f$ and $g$ must share the same set of preperiodic points (i.e., points of finite forward orbit under iteration). Unlike polynomial maps setting, one cannot appeal to the classification of polynomials with the same Julia set of \cite{BE87} and \cite{SS95}. Alternatively, we apply the Tits alternative for rational functions of \cite{BHPT} to deduce that $f$ and $g$ are not compositionally independent. 

\section{Preliminaries}
\subsection{Absolute values over global fields.} Here and throughout this section, $K$ denote a global field (unless otherwise stated) equipped with a set of inequivalent absolute
values (places) satisfying the product formula.  For example, $K$ could be a number field or a function field of  positive transcendence degree over
another field $k$. It was pointed out in \cite{Art06} that a global field which has an archimedean place must be a number field.\\
\indent If $K$ is a number field, we let $M_K$ to be the set of all absolute values of $K$ which extends the absolute values of $\mathbb{Q}$. If $K$ is a function field of finite transcendence degree $m\geq 1$ defined over a field $k$, then $M_K$ is the set of all absolute values on $K$ associated to the irreducible hypersurfaces of the projective space $\mathbb{P}^m_k$ (whose function field is $K':=k(t_1,...,t_m)$ where $t_1,..,t_m$ are algebraically independent--over $k$--functions in $K$). Thus there exist $N_v\geq 1$ such that the product formula $\prod_{v\in M_K}|x|_v^{N_v}=1$ holds for any nonzero $x\in K$. (see \cite{BG06} and \cite{La83}).\\
\indent Let $K_v$ be the completion of $K$ with respect to absolute value $|\cdot|_v$. Then $\mathbb{C}_v$ is the completion of the algebraic closure of $K_v$ and it is algebraically closed and complete. Note that, by abuse notation, we still denote the absolute value on $\mathbb{C}_v$ 
 which extends uniquely from $|\cdot|_v$ on $K$ by $|\cdot|_v$.

 \begin{rmk}
     When there is no confusion, we will simplify the notation to use $| \cdot |$ to denote the usual Archimedean norm on $K$.
 \end{rmk}
 
\subsection{Dynamical height and preperiodic points.} Let $S$ be the finite set of Galois conjugates $\mathrm{Gal}(\overline{K}/K)\cdot x$  of $x$ in $\mathbb{P}^1(\overline{K})$. Then we define the Weil height $h :\mathbb{P}^1(\overline{K})\rightarrow[0,+\infty)$ by 
$$h(x):=\frac{1}{|S|}\sum_{x\in S}\sum_{v\in M_K}N_v\log\max\{1,|x|_v\}.$$ In light of Baker-Rumely's construction \cite[$\S 7.9$]{BR10}, our Weil height is defined as the average over Galois conjugate of points in $\overline{K}$. It was pointed out that  this definition does not depend on the choice of embedding $\overline{K}\hookrightarrow\mathbb{C}_v$ as $S$ is Galois stable subset of $\mathbb{P}^1(\overline{K})$. \\ 
\indent Let $f$ be a rational function defined over a global field $K$ of degree $\geq2$. Then the Call-Silverman  canonical height $\hat{h}_f:\mathbb{P}^1(\overline{K})\rightarrow [0,+\infty)$ is defined by $$\hat{h}_f(x)=\lim_{n\rightarrow\infty}\frac{1}{(\mathrm{deg}\,f)^n}h(f^n(x)).$$ 
 We record important properties of Weil height and dynamical canonical  height as follows.
\begin{prop}\label{prop: canoheightproperties} (\cite[$\S3$]{Si07}, \cite{CS93}) Let $f$, $h$ and $\hat{h}_f$ be as above. Then
\begin{enumerate}
\item[(i)]  $h(f(x))=(\mathrm{deg}\,f)h(x)+O(1)$.
\item[(ii)] $\hat{h}_f(x)=h(x)+O(1).$
\item[(iii)] $\hat{h}_f(f(x))=(\mathrm{deg}\,f)\hat{h}_f(x).$
\end{enumerate}
\end{prop}
Over a number field, dynamical height is used to detect preperiodic points in the following sense.
\begin{prop} \label{prep: heightpreper} (\cite[Theorem 3.22]{Si07}) Let $f(x)\in K(x)$ be a rational function of degree $\geq 2$ defined over a number field $K.$ Then $x\in \mathbb{P}^1(\overline{K})$ is $f$-preperiodic if and only if $\hat{h}_f(x)=0.$
\end{prop}
The situation is different over function field. The dynamical height of a function defined on a function field (over an infinite field) fails to 
 properly detect preperiodic points. As an interesting example shown by Benedetto \cite{Be05}, we consider $f(x)=x^2\in K[x]$ defined over a function field $K=\overline{\mathbb{Q}}(T)$  of transcedence degree $1$ with constant field $\overline{\mathbb{Q}}$. Notice that all points in $\mathbb{P}^1(\overline{\mathbb{Q}})$ have dynamical height zero but $0,\infty$ and roots of unity in $\overline{\mathbb{Q}}$ are the only $f$-preperiodic points. In fact, $f$ is defined over the constant field of $K$. \\
 \indent Recall that a rational function $f(x)\in K(x)$ defined over a function field $K$ is called \textbf{isotrivial} if there exists a finite extension $K'/K$ and a M\"{o}bius transformation $\phi\in\mathrm{PGL}(2,K')$ such that $\phi^{-1}\circ f\circ \phi$ defined over the constant field of $K'.$ More related results, see \cite{Ba09}, \cite{Be05}, and \cite{PST09}.\\
 \indent After excluding maps with isotriviality, the dynamical height works as expected. The following result is due to Benedetto (for polynomial maps) and Baker (for rational functions) over $\mathbb{P}^1.$
\begin{thm}\label{thm: BakBenIsotrivia} (\cite{Ba09}, \cite{Be05}) Let $f(x)\in K(x)$ be a rational function of degree $\geq 2$ defined over a function field $K$. Suppose that $f$ is not isotrivial. Then $x\in \mathbb{P}^1(\overline{K})$ is $f$-preperiodic if and only if $\hat{h}_f(x)=0.$
\end{thm}
\subsection{Arakelov-Green's function.} The dynamical height associated to rational functions can be decomposed in terms of Arakelov-Green's function. We follow the construction in Baker-Rumely \cite[$\S10.3$]{BR10}.
Suppose that $f\in K(x)$ is a rational function of degree  $\geq2$ defined over a global field $K$. The $v$-adic Arakelov-Green's function of $f$ is a function $$G_{f,v}: \mathbb{P}^1_{\mathrm{Berk},v}\times\mathbb{P}^1_{\mathrm{Berk},v}\rightarrow\mathbb{R}\cup\{\infty\}$$ which is symmetric,  finite and continuous off the diagonal in $\mathbb{P}^1(\mathbb{C}_v)\times \mathbb{P}^1(\mathbb{C}_v)$. Here $\mathbb{P}^1_{\mathrm{Berk},v}$ is the Berkovich projective line, we refer the reader to \cite{Be19} and \cite{BR10} for detailed exposition. For each $y\in\mathbb{P}^1_{\mathrm{Berk},v}$, the function $G_{f,v}(x,y)$  is bounded differential variation on $\mathbb{P}^1_{\mathrm{Berk},v}$ and it  satisfies the differential equation 
\begin{equation}\Delta_x G_{f,v}(x,y)=\mu_{f,v}(x)-\delta_y(x)\label{eq:laplacian}\end{equation}
where $\mu_{f,v}$ is the equilibrium measure characterized by $f^*(\mu_{f,v})=(\mathrm{deg}\,f)\mu_{f,v}$ and $G_{f,v}$ is normalized so that $$
\iint_{\mathbb{P}^1_{\mathrm{Berk},v}\times \mathbb{P}^1_{\mathrm{Berk},v}}G_{f,v}(x,y)d\mu_{f,v}(x)d\mu_{f,v}(y)=0.$$ Our Laplacian is the negative of the one introduced by Baker-Rumely \cite[$\S3.5$]{BR06}. For any $x,y\in \mathbb{P}^1(\overline{K})$ such that $x\neq y$, we have
$$\hat{h}_f(x)+\hat{h}_f(y)=\sum_{v\in M_K}N_vG_{f,v}(x,y)$$ where the positive number $N_v$ is the same normalization as appeared in the product formula.
\subsection{Results from Diophantine geometry.} 
Here we collect some Diophantine geometry results that will be used in the proof.
\begin{thm}\cite[Theorem 6.1.3]{EG15}\label{thm: S-unit}
    Let $K$ be a field of characteristic $0$, let $n \geq 2$, let $a_1, \dots,a_n \in K^*$ and let $\Gamma$ be a subgroup of $(K^*)^n$ of finite rank $r$. Then the number of non-degenerate solutions of
    $$ a_1 x_1 + a_2 x_2 + \dots + a_n x_n = 1$$
    in $(x_1, x_2, \dots, x_n) \in \Gamma$ is finite.
\end{thm}
\begin{rmk}
    We call a solution to the equation $a_1 x_1 + a_2 x_2 + \dots + a_n x_n = 1$ like above non-degenerate if 
    $$ \sum_{i \in I}a_i x_i \neq 0$$
    for any proper subset $I \subset \{1,2, \dots, n\}$.
\end{rmk}

\subsection{The Ping-pong Lemma}
\begin{lem}\label{lem: ping-pong}
 Let $G$ be a group of maps under compositions. Let $S = \langle f_1, \dots, f_r \rangle$, $r \geq 2$ and $f_i \in G$, for $i \in \{1,2, \dots, r\}$, be a semigroup of $G$. If there exists a collection of non-empty mutually non-intersecting sets $I_1, \dots, I_r$ such that $f_i(\bigcup^r_{j = 1}I_j)= I_i$ for all $1 \leq i \leq r$, then $S$ is a free semigroup of rank $r$
with free basis $f_1,\dots,f_r$.   
\end{lem}

\begin{proof}
    This is basically the Ping-pong Lemma stated in \cite[The Ping-Pong Lemma]{KT21}. The only difference is that it is stated for $G$ being the group of affine linear maps on $\R$. However the proof doesn't rely on this assumption.
\end{proof}
\subsection{Asymptotic notation} 
We clarify the asymptotic notation here:
\begin{defn}
    Consider two functions $f(n)$ and $g(n)$ with $n \in \N$ and take values in $\C$. We defined the asymptotic notation when $n$ goes large as follows:
    \begin{enumerate}
        \item We write $f(n) \sim g(n)$ if there exists a non-zero constant $k \in \C$ such that 
    $$ \lim_{n \to \infty}f(n)/g(n) = k;$$
    \item We write $f(n) = o(g(n))$ if 
    $$ \lim_{n \to \infty} f(n)/g(n) = 0;$$
    \item We write $f(n) = O(g(n))$ if 
    $$ \limsup_{n \to \infty} f(n)/g(n) < \infty.$$
    \end{enumerate}
    Furthermore, for two functions $f(Z)$ and $g(Z)$ take values in $\C$ with $Z \in \C$. We define the asymptotic notation when $Z$ is close to $0$ similarly as above but taking limits with respect to $Z \to 0$. 
    
\end{defn}
\section{Proof of Theorem \ref{thm: mainlinear}}

\begin{lem} \label{lem: reduce-to-sharingfixpoint}
    Let $f(X)$, $g(X)$ be automorphisms on $\P^1$ defined over $\C$ and $c(X)$ be a rational function defined over $\C$. If there are infinitely many $n \in \N$ such that 
    \begin{equation} \label{eq: targeteq1}
        c(X) = f^n(X) = g^n(X)
    \end{equation}
    has solutions, then $f$ and $g$ share a common fixed point unless $f(X)$, $g(X)$ are simultaneously conjugated to $\alpha X + \beta$, $X/(\gamma X + \delta)$ respectively, with some $\alpha, \delta, \beta ,\gamma \in \C^*$, such that at least one of the following conditions is satisfied:
    \begin{enumerate}
        \item at least one of  $\alpha$, $\delta$ is root of unity other than $1$;
        \item one of $\alpha/\delta$ and $\alpha \delta$ is a root of unity.
    \end{enumerate}
   
\end{lem}
\begin{rmk} We record important notes in order:
  \begin{enumerate}\item[(i)]  For the rest of the discussion, we will use the notation $c(X) = A(X) /B(X)$ for a pair of coprime $A(X), B(X) \in \C[X]$. 
  \item[(ii)] Recall that the field of  Puiseux series with complex coefficients $\mathbb{C}$ is defined as the union 
  $$\mathcal{P}(Z):=\bigcup_{n\geq 1}\mathbb{C}(\!(Z^{1/n})\!)=\left\{\sum_{m=m_0}^{\infty}c_mZ^{m/n}: c_m\in\mathbb{C}, m_0\in\mathbb{Z}, n\geq 1 \right\}$$ of the fields of Laurent series in (indeterminate) $Z^{1/n}$. Then we define $\val_Z$ on $\mathcal{P}(Z)\backslash \{0\}$ by letting $$\val_Z(F)=\min\left\{\frac{m}{n}: c_m\neq0\right\}$$ for $F\in \mathcal{P}(Z)\backslash\{0\}.$ For example,  $\val_Z(Z^3-2Z^{5})=3.$
  
\end{enumerate}
\end{rmk}
\begin{proof}
     Suppose $f$ and $g$ don't share a fixed point, otherwise we have nothing to prove. After conjugation, we assume that 
    $$ f(X) = \alpha X + \beta, $$
    $$ g(X) = \frac{X}{\gamma X + \delta}.$$ Notice that if $\gamma$ or $\beta$ is zero, then certainly $f(X)$ and $g(X)$ share a common fixed point. So we assume $\beta$ ,$\gamma$ are non-zero. 
    We also have that $\alpha$ and $\delta$ are non-zero as otherwise it will result in a constant map. We break the discussion into two cases depending on $\alpha$ and $\delta$. \newline

    \noindent\textbf{Case \RNum{1}.} Suppose $\alpha$ is not a root of unity and $\delta = 1$. 
    \begin{lem} \label{lem: lemCase1}
     There are only finitely many $n \in \N$ such that   $$
        c(X) = f^n(X) = g^n(X)
    $$ 
    has solutions.
    \end{lem}
    The same result for $\alpha = 1$ and $\delta$ is not a root of unity holds. We just need to further conjugate $c(X)$, $f(X)$ and $g(X)$ by $1/X$ to swap the role of $\alpha$ and $\delta$. Then the same argument will give the result. This concludes these cases.
   \newline

    \noindent\textbf{Case \RNum{2}.} Now, suppose both $\alpha$ and $\delta$ are not roots of unity and neither $\alpha/\delta$ nor $\alpha \delta$ is a root of unity.
    Then let 
    $$X^\pm_n = \left(-t_n  \pm \sqrt{t^2_n - 4 \alpha^n \delta^n \frac{(1- \alpha^n)(1- \delta^n  )}{(1 - \alpha)(1-\delta)}\beta \gamma}\right)/\left(2\alpha^n \frac{1 - \delta^n}{1 - \delta}\gamma\right),$$
    be the solution to the $f^n(X) = g^n(X)$, where 
    $$ t_n = \frac{(1-\alpha^n)(1 - \delta^n)}{(1- \alpha)(1- \delta)}\beta \gamma + \alpha^n \delta^n - 1.$$
    
    We break the proof into steps: 
    \begin{itemize}
        \item \textbf{Step \RNum{1}}: we will show that $c(X) = f^n(X) = g^n(X)$ has solution for infinitely many $n$ will pose a restriction on $\alpha$ and $\delta$: there exists a non-trivial polynomial in two variables $F(X,Y)$ such that $F(\alpha^n, \delta^n) = 0$ holds for infinitely many $n \in \N$. 
        \item  \textbf{Step \RNum{2}}: we use the S-unit Theorem to argue that this implies $\alpha$ and $\delta$ are multiplicatively dependent. 
        \item \textbf{Step \RNum{3}}: we deduce that then $c(X) = f^n(X) = g^n(X)$  have solution for infinitely many $n$ if $f(X)$ and $g(X)$ share a common fixed point.
    \end{itemize}

   \noindent\textbf{Step \RNum{1}.} Suppose there are infinitely many $n \in \N$ such that either
    \begin{equation}\label{eq: case4nontrivialeq1}
        c(X^-_n) = f^n(X^-_n)
    \end{equation}
    or 
   \begin{equation}\label{eq: case4nontrivialeq2}
       c(X^+_n) = f^n(X^+_n).
   \end{equation} 
    
    We claim that the following equation, which is obtained from above by replacing $\alpha^n$ with a variable $Z$ and $\delta^n$ with a variable $W$,
    \begin{equation}\label{eq: rationalcase4eq1}
        c(X^\pm(Z,W))  = ZX^\pm(Z,W)  + \frac{1 - Z}{1 - \alpha}\beta
    \end{equation}
    
    doesn't hold constantly for every $Z,W \in \C$, where 
    \begin{equation}\label{eq: expression-xpm}
        X^\pm(Z,W) = \left(-t(Z,W)  \pm \sqrt{t(Z,W)^2 - 4 Z W \frac{(1- Z)(1- W)}{(1 - \alpha)(1-\delta)}\beta \gamma}\right)/\left(2Z \frac{1 - W}{1 - \delta}\gamma\right),
    \end{equation} 
    $$ t(Z,W) = \frac{(1-Z)(1 - W)}{(1- \alpha)(1- \delta)}\beta \gamma + Z W - 1.$$
    Therefore, we get the desired result of Step \RNum{1} as Equation (\ref{eq: rationalcase4eq1}) will give a non-trivial restriction on $\alpha^n$ and $\delta^n$.

    To prove the claim, taking a specialization at $W = 0$, we have 
    $$ \restr{X^+(Z, W)}{W = 0}  = \frac{-t(Z,0) + t(Z,0)}{2Z\gamma/(1 - \delta)}= 0,$$
    and Equation (\ref{eq: rationalcase4eq1}) implies
    $$ c(0) = \frac{1 - Z}{1 - \alpha} \beta$$
    which only holds for finitely many values of $Z$. So, 
    $$c(X^+(Z,W))  = Z(X^+(Z,W))  + \frac{1 - Z}{1 - \alpha}\beta$$
    doesn't hold constantly for every value of $Z$ and $W$.
    
    On the other hand, we have 
    $$  X^-_0(Z) \vcentcolon = \restr{ X^-(Z, W)}{W = 0} = \frac{-2t(Z,0)}{2Z\gamma/(1 - \delta)} = \frac{\beta}{1 - \alpha} + \left(\frac{1 - \delta}{\gamma} - \frac{\beta}{1 - \alpha}\right)\frac{1}{Z}.$$
    If 
    $$ \frac{1 - \delta}{\gamma} = \frac{\beta}{1 - \alpha}$$
    then
    $f(X)$ and $g(X)$ share a common fixed point and we are done. So, we assume 
    $$ \frac{1 - \delta}{\gamma} \neq \frac{\beta}{1 - \alpha}.$$
    Then Equation (\ref{eq: rationalcase4eq1}) implies
    \begin{equation}\label{eq: rationalcase4contra1}
        c(X^-_0(Z))= A(X^-_0(Z))/B(X^-_0(Z))  = ZX^-_0(Z) + \frac{1 - Z}{1 - \alpha}\beta  = \frac{1- \delta}{\gamma}.
    \end{equation} 
    This implies that there are only finitely many $Z \in \C$ satisfies the condition.
    
    Thus $$ c(X^-(Z,W)) = f^n(X^-(Z,W))$$
    doesn't hold constantly.
    \newline
    
    \noindent\textbf{Step \RNum{2}.} Now, either 
    $$c(X^{+}(Z,W)) = Z(X^+(Z,W))  + \frac{1 - Z}{1 - \alpha}\beta $$
    or 
    $$ c(X^-(Z,W)) = Z(X^-(Z,W))  + \frac{1 - Z}{1 - \alpha}\beta$$
    is a non-trivial equation on variables $Z$ and $W$ with $(Z,W) = (\alpha^n, \delta^n)$ as solutions for infinitely many $n \in \N$.
    Thus, this gives a polynomial, by clearing the denominators and possible square roots on both sides of the equation, 
    \begin{equation}\label{eq: applyS-unitequationrational}
        \sum_{1 \leq i \leq k_1, 1 \leq j \leq k_2} c_{i,j}\alpha^{in}\delta^{jn} = 0
    \end{equation}
    for infinitely many $n \in \N$. 
    Then, we want to apply the S-unit theorem (Theorem \ref{thm: S-unit}) to show that $\alpha$ and $\delta$ are multiplicatively dependent. Let $G \subset \C^*$ be the finitely generated subgroup generated by $\alpha$ and $\delta$. Then Equation (\ref{eq: applyS-unitequationrational}) can be rewritten as 
    \begin{equation}\label{eq: rewrittenS-unit}
        \sum_{1 \leq i \leq m}c_{i}W_i = 0,
    \end{equation}
    for some positive integer $m$, $c_i \in \C^*$ and $W_i \in G$ for each $i = 1, 2, \dots, m$, which has solutions $$(W_1, W_2, \dots, W_m)=(\alpha^{s_1n}\delta^{t_1n}, \alpha^{s_2n}\delta^{t_2n}, \dots, \alpha^{s_mn}\delta^{t_mn}) \in G^m,$$ for infinitely many $n \in \N$, where $(s_i, t_i)$'s are distinct pairs of integers. The S-unit theorem \ref{thm: S-unit} says that there are only finitely many non-degenerate solutions to Equation (\ref{eq: rewrittenS-unit}). After an easy induction argument on the numbers of terms involved in Equation (\ref{eq: applyS-unitequationrational}) we can get that there exists infinitely many $n \in \N$ such that 
    $$ c_{i_1} \alpha^{s_{i_1}n}\delta^{t_{i_1}n} + c_{i_2}\alpha^{s_{i_2}n}\delta^{t_{i_2}n} = 0,$$
    where $i_1$ and $i_2$ are distinct and in $\{1,2, \dots, m\}$. Since $\alpha$ and $\delta$ are not zero and not roots of unity, this is only possible if $s_{i_1} \neq s_{i_2}$ and $t_{i_1} \neq t_{i_2}$ and $c_{i_1}/c_{i_2}$ is a root of unity. This implies that $\alpha$ and $\delta$ are multiplicative dependent. 
    In another word, there exist non-zero integers $k_1$ and $k_2$ such that $\alpha^{k_1} = \delta^{k_2}$.
    \newline
     
    \noindent\textbf{Step \RNum{3}.} Let $k = k_1/k_2 \in \Q^*$ such that $\delta= \xi \alpha^k$, where $\xi$ is a root of unity of order $k_2 \in \N$ and $\alpha^k$ is understood as a particular choice of $k_2$-th root of $\alpha^{k_1}$. Then our assumption implies that there exists an $i \in \{1,2, \dots, k_2\}$ such that either the system of equations 
    \begin{equation}\label{eq: rationalcaseformal1}
        A(X^+(Z, \xi^i Z^k)) = B(X^+(Z,\xi^iZ^k))\left(Z X^{+}(Z,\xi^iZ^k) + \frac{1 - Z}{ 1- \alpha}\beta\right),
    \end{equation}
    \begin{equation}\label{eq: rationalcaseformal3}
       A(X^+(Z, \xi^i Z^k))\left(\frac{1- \xi^i Z^k}{1- \delta}\gamma X^+(Z, \xi^i Z^k) + \xi^iZ^k\right) = B(X^+(Z, \xi^i Z^k))X^+(Z, \xi^i Z^k)
    \end{equation}
    or the system of equations
   \begin{equation} \label{eq: rationalcaseformal2}
       A(X^-(Z, \xi^i Z^k)) = B(X^-(Z,\xi^iZ^k))\left(Z X^{-}(Z,\xi^iZ^k) + \frac{1 - Z}{ 1- \alpha}\beta\right),
   \end{equation}
   \begin{equation}\label{eq: rationalcaseformal4}
       A(X^-(Z, \xi^i Z^k))\left(\frac{1- \xi^i Z^k}{1- \delta}\gamma X^-(Z, \xi^i Z^k) + \xi^i Z^k\right) = B(X^-(Z, \xi^i Z^k))X^-(Z, \xi^i Z^k)
   \end{equation}
    has infinitely many values of $Z$, given by $Z = \alpha^{nk_2 + i}$ for infinitely many $n \in \N$, as solutions. This cannot happen unless the system of equations holds constantly for every possible value of $Z$ as it only has one variable. From now on, we denote 
    $$ X^\pm_i(Z) \vcentcolon = X^\pm(Z, \xi^i Z^k).$$
    Notice that 
    \begin{equation} \label{eq: rationalproductsofroots1}
      X^+_i (Z) X^-_i(Z) = \frac{\xi^iZ^k(1 - Z)(1 - \delta)\beta}{Z(1-\xi^i Z^k)(1 - \alpha)\gamma},
    \end{equation}
    \begin{equation}
        X^+_i(Z) + X^-_i(Z) = -2t(Z, \xi^i Z^k)\Big/ \left(2Z \frac{1 - \xi^i Z^k}{1-\delta}\gamma\right),
    \end{equation} 
    for every $i \in \{1,2, \dots, k_2\}$.

    We also notice that, from the Equation (\ref{eq: expression-xpm}), if $k > 0$, we have 
    \begin{align}\label{eq: exk<0-}
        & X^-_i(Z) = \left(-\left( \frac{\beta \gamma}{(1 - \alpha)(1 - \delta)} -1\right)  + o(1) - \sqrt{\left(\frac{\beta \gamma}{(1 - \alpha)(1 - \delta)}-1\right)^2 + o(1)}\right) \nonumber \\
        & \times \left( 2\frac{\gamma}{1 - \delta} Z + o(Z)\right)^{-1} \nonumber \\
        & = \left( -2\left( \frac{\beta \gamma}{(1 - \alpha)(1 - \delta)} -1\right) + o(1)\right)\times \left( 2\frac{\gamma}{1 - \delta} Z + o(Z)\right)^{-1} \nonumber \\  
         & = \frac{1-\delta}{2\gamma}Z^{-1}\left( - 2\frac{\beta \gamma}{(1- \alpha)(1 - \delta)} + 2 + o(1) \right),
   \end{align}
    when $Z$ is around $0$, which in particular implies that $\val_Z(X^-_i(Z)) = -1$ as 
    $$ \frac{\beta}{1 - \alpha} \neq \frac{1 - \delta}{\gamma},$$
    since otherwise it implies $f(X)$ and $g(X)$ share a common fixed point. Then from Equation (\ref{eq: rationalproductsofroots1}), we have 
    \begin{equation}\label{eq: valuation-of-X+>0}
        \val(X^+_i(Z)) = k.
    \end{equation}
    \newline 

    \noindent\textbf{Subcase \RNum{1}.}  Let's first suppose $k > 0$ and the system of Equation (\ref{eq: rationalcaseformal1}) and (\ref{eq: rationalcaseformal3}) have infinitely many zeros. That is saying they hold constantly for every $Z$. Then Equation (\ref{eq: rationalcaseformal1}) implies that 
    \begin{equation} \label{eq: case4contralk}
        A(X^+_i(Z))/B(X^+_i(Z)) = \frac{\beta}{1 - \alpha} - \frac{\beta}{1 - \alpha} Z + o(Z),
    \end{equation}
    when $Z$ is around $0$.
    However, if $k > 1$, we will have 
    $$ A(X^+_i(Z))/B(X^+_i(Z)) = A(0)/B(0) + O(Z^k),$$
    which contradicts Equation (\ref{eq: case4contralk}). If $0 < k < 1$, we can conjugate $f(X)$, $g(X)$ and $c(X)$ by $1/X$ to swap the role of $\alpha$ and $\delta$ and we get $k > 1$ again. The same argument shows that this case is also not possible. Thus, we left with the case $k = 1$. However, our assumption implies that $k \neq 1$ since otherwise $\alpha/\delta$ is a root of unity. So we are done in this case.
    \newline 

     \noindent\textbf{Subcase \RNum{2}.} Now let's suppose $k < 0$ and the system of Equation (\ref{eq: rationalcaseformal1}) and (\ref{eq: rationalcaseformal3}) have infinitely many zeros. Then Equation (\ref{eq: expression-xpm}) implies that 
     \begin{align}
         & X^+_i (Z) =  \left(-\left(-\frac{\beta \gamma}{(1 - \alpha)(1 - \delta)}\xi^iZ^k\right) + o(Z^k) + \sqrt{\left(-\frac{\beta \gamma}{(1 - \alpha)(1 - \delta)}\right)^2\xi^{2i}Z^{2k} + o(Z^{2k}) }\right) \nonumber \\
         &\times \left(-2\frac{\gamma}{1 - \delta} \xi^i Z^{k+1} + o(Z^{k+1})\right)^{-1} \nonumber \\
         & = 2 \left(\frac{\beta \gamma}{(1 - \alpha)(1 - \delta)}\xi^i Z^k + o(Z^k)\right) \times \left(-2\frac{\gamma}{1 - \delta}\xi^i Z^{k+1} + o(Z^{k+1})\right)^{-1} \nonumber \\
         & = -\frac{1-\delta}{2\gamma \xi^i}Z^{-1 -k}\left(2 \frac{\beta \gamma}{(1- \alpha)(1 -\delta)} \xi^iZ^k + o(Z^k)\right),
     \end{align} 
    when $Z$ is around $0$ and in particular we have that 
    $$ \val_Z(X^+_i(Z)) = -1.$$
    Now, if $-1< k < 0$, we have Equation (\ref{eq: rationalcaseformal3}) implies that 
\begin{align}
   & - \deg(A) -1 + k = \val_{Z}\left(A(X^+_i(Z))\left(\frac{1- \xi^i Z^k}{1- \delta}\gamma X^+_i(Z) + \xi^iZ^k\right)\right)   \nonumber \\ 
   &= \val_Z(B(X^{+}_i(Z))X^+_i(Z) ) = -\deg(B)-1,
\end{align}
which is impossible as $k$ is not an integer. This again gives a contradiction. Now, similarly by the conjugation argument we also know that $k < -1$ is also impossible. Moreover, by the assumption we have $k \neq -1$ since otherwise $\alpha \delta $ is a root of unity. Thus, we conclude the proof for $k < 0$ and Equation (\ref{eq: rationalcaseformal1}) and (\ref{eq: rationalcaseformal3}) has infinitely many zeros.
\newline

\noindent\textbf{Subcase \RNum{3}.} Let's now suppose $k > 0$ and the system of Equation (\ref{eq: rationalcaseformal2}) and (\ref{eq: rationalcaseformal4}) have infinitely many zeros. We calculate similarly as in Equation (\ref{eq: exk<0-}), when $Z$ is close to $0$, but we assume first that $0 < k < 1$ to get 
the following: 
$$ t(Z, \xi^i Z^k) = \K_1 - 1 - \K_1 \xi^i Z^k  + o(Z^k)$$
$$ t(Z, \xi^iZ^k)^2 = (\K_1 - 1)^2 - 2 \K_1 (\K_1 - 1) \xi^i Z^k + o(Z^k)$$
\begin{align}
    \sqrt{ t(Z,\xi^iZ^k)^2 - 4\K_1\xi^i Z^{k+1}(1 - Z)(1- \xi^iZ^k) } &=& \sqrt{(\K_1 -1)^2 - 2\K_1(\K_1 -1)\xi^iZ^k + o(Z^k)} \nonumber\\ &=& \K_1 - 1 - \K_1 \xi^i Z^k + o(Z^k), \nonumber
\end{align}
where $\K_1 = \beta\gamma/((1- \alpha)(1 - \delta))$.
Therefore, Equation (\ref{eq: expression-xpm}) gives us 
\begin{align}
    X^-_i(Z) &= (-2(\K_1 -1) + 2 \K_1\xi^iZ^k + o(Z^k))(2\K_2 Z)^{-1}(1 + \xi^iZ^k + o(\xi^i Z^k)) \nonumber \\
    &= -\frac{\K_1 -1}{\K_2} Z^{-1} + \frac{\K_1}{\K_2} \xi^i Z^{k-1} - \frac{\K_1 -1}{\K_2} \xi^iZ^{k-1} + o(Z^{k-1}) \nonumber\\ 
    &= -\frac{\K_1 -1}{\K_2} Z^{-1} + \frac{1}{\K_2} \xi^i Z^{k-1} + o(Z^{k-1})  ,
\end{align}
where $\K_2 = \gamma/(1 - \delta)$. Notice that this is just a further expansion of the expression (\ref{eq: exk<0-}) when $0 < k < 1$. Since $\val_Z(X^-_i(Z)) = -1$ and 
$$ ZX^{-}_i(Z) + \frac{\beta}{1 - \alpha} = \K^{-1}_2 + o(1),$$ Equation (\ref{eq: rationalcaseformal2}) implies that $\deg(A) = \deg(B) = m$ for some positive integer $m$. Then the left hand side of Equation (\ref{eq: rationalcaseformal2}) is 
\begin{align}
&A\left( -\frac{\K_1 -1}{\K_2} Z^{-1} + \K^{-1}_2 \xi^i Z^{k-1} + o(Z^{k-1})\right) \nonumber \\
&= a_m \left(  -\frac{\K_1 -1}{\K_2} Z^{-1} + \K^{-1}_2 \xi^i Z^{k-1} + o(Z^{k-1})\right)^m + O(Z^{-m+1}) \nonumber\\
  &=  a_m \left(- \frac{\K_1-1}{\K_2}\right)^m Z^{-m} + a_m m\left(-\frac{\K_1 -1}{\K_2}\right)^{m-1}\K^{-1}_2\xi^iZ^{-m +k} + o(Z^{-m+k}),
\end{align}
where $a_m$ is a non-zero constant.
The right hand side of Equation (\ref{eq: rationalcaseformal2}) is 
\begin{align}
&B\left(  -\frac{\K_1 -1}{\K_2} Z^{-1} + \K^{-1}_2 \xi^i Z^{k-1} + o(Z^{k-1}) \right)\left( -\frac{\K_1 -1}{\K_2}  + \K^{-1}_2 \xi^i Z^{k} + o(Z^{k}) + \K_3  - \K_3 Z \right) \nonumber \\
&= b_m\left( -\frac{\K_1 -1}{\K_2} Z^{-1} + \K^{-1}_2 \xi^i Z^{k-1} + o(Z^{k-1})\right)^m\left(\K_3-\frac{\K_1 - 1}{\K_2} + \K^{-1}_2\xi^iZ^k + o(Z^k) \right) + O(Z^{-m + 1}) \nonumber\\
&= b_m\left(\K_3 - \frac{\K_1-1}{\K_2} \right )\left( -\frac{\K_1 -1}{\K_2}\right)^m Z^{-m} + b_m\left(- \frac{\K_1-1}{\K_2} \right)^m\K^{-1}_2\xi^iZ^{-m+k} \nonumber\\&+ b_m m\left(-\frac{\K_1-1}{\K_2}\right)^{m-1}\K^{-1}_2\left(\K_3 - \frac{\K_1 -1}{\K_2}\right)\xi^iZ^{-m +k} + o(Z^{-m + k})\nonumber \\
    &=b_m\left(\K_3 - \frac{\K_1-1}{\K_2} \right )\left( -\frac{\K_1 -1}{\K_2}\right)^m Z^{-m} + b_m(m+1)\left( - \frac{\K_1 -1}{\K_2}\right)^m\K^{-1}_2 \xi^i Z^{-m+k} \nonumber \\&+ b_m m \K_3 \left( -\frac{\K_1 -1}{\K_2}\right)^{m-1}\K^{-1}_2\xi^i Z^{-m + k} + o(Z^{-m + k}),
\end{align}
where $\K_3 = \beta/(1- \alpha)$ and $b_m$ is a non-zero constant. Notice by definition that $\K_1=\K_2\K_3$, we obtain 
$$ -\frac{\K_1 - 1}{\K_2} = \K_2^{-1} - \K_3 \neq 0,$$
and 
$$ \K^{-1}_2 \neq 0.$$
Therefore, looking at the coefficients of $Z^{-m}$ terms of Equation (\ref{eq: rationalcaseformal2}) gives us 
$$ a_m = \K^{-1}_2b_m.$$
Then, examining the coefficients of $Z^{-m+k}$ terms of Equation (\ref{eq: rationalcaseformal2}) gives us
\begin{align}
  \K^{-1}_2b_m m\left(-\frac{\K_1 -1}{\K_2}\right)^{m-1}\K^{-1}_2 \nonumber\\=  b_m(m+1)\left( - \frac{\K_1 -1}{\K_2}\right)^m\K^{-1}_2 +  b_m m \K_3 \left( -\frac{\K_1 -1}{\K_2}\right)^{m-1}\K^{-1}_2,  
\end{align}
which is equivalent to
\begin{align}
    m(\K_2^{-1} - \K_3)b_m \left( \K_2^{-1} - \K_3\right)^{m-1} \K^{-1}_2  = (m+1)b_m(\K_2^{-1} - \K_3)^m\K^{-1}_2.
\end{align}
This is a contradiction as $m \neq m +1$.

    Similarly, if $k > 1$, we conjugate $f(X)$, $g(X)$ and $c(X)$ by $1/X$ to swap the role of $\alpha$ and $\delta$ and we get $0 < k < 1$ again. The same argument shows that this case is also not possible. Thus, we left with the case $k = 1$. However, our assumption implies that $k \neq 1$ since otherwise $\alpha/\delta$ is a root of unity. So we are done in this case.  
     \newline

\noindent\textbf{Subcase \RNum{4}.} Let's now suppose $k < 0$ and the system of Equation (\ref{eq: rationalcaseformal2}) and (\ref{eq: rationalcaseformal4}) have infinitely many zeros.  We first suppose $k < -1$. To ease the notation, we denote 
    $$ \K \vcentcolon = \frac{\beta \gamma}{(1 - \alpha)(1 - \delta)}.$$
    Notice that when $|Z| \to \infty$, we have
    \begin{equation}
        t(Z, \xi^iZ^k) =  - \K Z  + (\K -1) + (1 + \K)\xi^iZ^{k+1}- \xi^iZ^k \K,
    \end{equation}
    \begin{equation}
        t^2(Z, \xi^iZ^k) = \K^2  Z^{2}  -2 \K(\K - 1) Z - 2 \K (\K +1)\xi^i Z^{k+2}  + o(Z^{k+2}),
    \end{equation}
    \begin{equation}
        4\xi^i Z^{k+1} (1 - Z)(1 - \xi^iZ^k)\K = - 4 \K \xi^i Z^{k+2} + o(Z^{k+2}).
    \end{equation}
     Thus, by Equation (\ref{eq: expression-xpm}), we have 
    \begin{align}
        & X^-_i(Z) = \left(- t(Z, \xi^i Z^k) -  \right. \nonumber \\
        & \left. \sqrt{ 
            \K^2  Z^{2}  - 2\K (\K -1) Z- 2 \K (\K -1)\xi^i Z^{k + 2}  +  o(Z^{k+2})  
        } \right) \nonumber \\
        & \times \left(- 2 \xi^i Z^{k+1} \frac{\gamma}{ 1- \delta} + 2 Z \frac{\gamma}{ 1- \delta}\right)^{-1},
    \end{align}
    and, breaking the square root, we have that this equals to
    \begin{align}
        & = \left(-t(Z, \xi^iZ^{k})  - \K Z + (\K - 1) + (\K -1) \xi^i Z^{k+1} + o(Z^{k+1})\right) \nonumber \\
        & \times \left(- 2 \xi^i Z^{k+1} \frac{\gamma}{ 1- \delta} + 2 Z \frac{\gamma}{ 1- \delta}\right)^{-1}
    \end{align}
    \begin{equation}
         = \frac{- 2 \xi^i Z^{k+1} + o(Z^{k+1})}{2 \gamma/(1 - \delta)} Z^{-1}\sum^{\infty}_{j = 0} \xi^{ij} Z^{kj}
    \end{equation}
    \begin{equation}
         = -\frac{1 - \delta}{\gamma}\xi^i Z^k+ o(Z^{k}).
    \end{equation}

    Now, Equation (\ref{eq: rationalcaseformal2}) implies that 
    \begin{equation}\label{eq: A2midstepk<0}
        A\left(-\frac{1 - \delta}{\gamma}\xi^i Z^k + o(Z^k)\right) = B\left(-\frac{1 - \delta}{\gamma}\xi^i Z^k + o(Z^k)\right)\left( -\frac{\beta}{1 - \alpha}Z + o(Z) \right),
    \end{equation}
    when $|Z| \to \infty$. Since $A(X)$ is a polynomial, the left hand side is bounded as $|Z|$ increasing. Therefore, 
    to make the right hand side of Equation (\ref{eq: A2midstepk<0}) bounded, we need 
    $$ B(X) = X^m B_1(X)$$
    for some positive integer $m$ and polynomial $B_1(X)$ with non-trivial constant term $b_m$. Thus, we have
    \begin{equation}
        B\left(-\frac{1 - \delta}{\gamma}\xi^i Z^k + o(Z^k)\right) = b_m \left(- \frac{1 - \delta}{\gamma}\right)^m\xi^{im} Z^{mk} + o(Z^{mk}).
    \end{equation}
     In particular, $B(0) = 0$. 
    
    However, since $k < -1$, we have $mk + 1 < 0$. Therefore, the right hand side of Equation (\ref{eq: A2midstepk<0}) also goes to $0$ when $|Z|$ approaches to infinity. This implies that $A(0) = 0$, which contradicts that $A(X)$ and $B(X)$ don't share a common zero. 
    Notice that when $-1 < k < 0$, we can swap the role of $\alpha$ and $\delta$ by conjugating $f(X)$, $g(X)$ and $c(X)$ by the map $1/X$ just like what we did above. Then we are back to $k < -1$ case. Also, by the assumption we have $k \neq -1$ as otherwise it will imply $\alpha \delta$ is a root of unity.

\end{proof}

\begin{proof}[Proof of Lemma \ref{lem: lemCase1}]
    Suppose there are infinitely many $n \in \N$ such that 
    $$
        c(X) = f^n(X) = g^n(X)
    $$
    has solutions. Let's denote 
    $$X^{\pm} _n = \frac{-t_n \pm \sqrt{t^2_n - 4\alpha^nn\gamma(1 - \alpha^n)\beta/(1 - \alpha)}}{2n\alpha^n\gamma},$$
    $$ t_n = \alpha^n + n\gamma \frac{1- \alpha^n}{1- \alpha}\beta -1,$$
    where $X^\pm_n$ satisfies $f^n(X^\pm_n) = g^n(X^\pm_n)$.
    
    Let's first suppose $|\alpha|> 1$ and there are infinitely many $n \in \N$ such that
\begin{equation}\label{eq: nosharefixeq1}
        c(X^+_n) = f^n(X^+_n) 
    \end{equation}
    holds. Let's denote $c(X) = A(X)/B(X)$ where $A(X)$ and $B(X)$ are two polynomials. 
    Since 
    \begin{align}
        t^2_n - 4 \alpha^nn\gamma \frac{1 - \alpha^n}{1 - \alpha} \beta &=& \left(\frac{n\gamma\alpha^n \beta}{1 - \alpha}\right)^2 + \frac{2n\gamma\beta \alpha^{2n}}{1- \alpha} + \alpha^{2n}  -\frac{2n^2\gamma^2\alpha^n\beta^2}{(1-\alpha)^2} \nonumber\\ &-&\frac{2\gamma\beta}{1 -\alpha}n\alpha^n - 2\alpha^n  + \frac{\gamma^2\beta^2}{(1- \alpha)^2}n^2- \frac{2n\gamma\beta}{1- \alpha} + o(n),
    \end{align} 
    we have 
    \begin{align} \label{eq: xpmnexpansiond=1}
        & X^\pm_n=  \left(\frac{\gamma \beta}{1- \alpha}n\alpha^n -\alpha^n - \frac{\gamma\beta}{1-\alpha}n + 1 \pm \left(- \frac{\gamma\beta}{1 - \alpha}n\alpha^n - \alpha^n + \frac{n\gamma\beta}{1 - \alpha} + 1 - 2\alpha^{-n}+ o(\alpha^{-n})\right)\right) \nonumber \\ & \times (2n\alpha^n \gamma)^{-1}.
    \end{align}
    Therefore,
    $$\lim_{n \to \infty} |X^+_n| =  0,$$
    $$ \lim_{n \to \infty} |\alpha^n X^+_n| = \infty.$$
    Since we have 
    $$X^+_n \sim 1/n$$
    when $n$ goes large, we have similarly
    $$A(X^+_n)/B(X^+_n) \sim n^l,$$
    $$ \alpha^nX^+_n + \frac{1 - \alpha^n}{1- \alpha}\beta \sim \alpha^n$$
    when $n$ is large with some integer $l$, which contradicts the Equation (\ref{eq: nosharefixeq1}).

    Suppose $|\alpha| > 1$ and there are infinitely many $n \in \N$ such that 
    \begin{equation} \label{eq: nosharefixeq2}
        c(X^-_n) = f^n(X^-_n)
    \end{equation}
    holds. Notice that 
    $$ X^-_n = \frac{\beta}{1 - \alpha} - \frac{\beta}{1 - \alpha}\alpha^{-n} + \gamma^{-1}n^{-1}\alpha^{-2n} + o(n^{-1}\alpha^{-2n}) ,$$
    when $n$ becomes large.
    Therefore, 
    $$f^n(X^-_n) \sim n^{-1}\alpha^{-n}$$
    $$ A(X^-_n)/B(X^-_n) \sim \alpha^{-nm}$$
    for some integer $m$ when $n$ becomes large.
    These contradict the assumption that Equation (\ref{eq: nosharefixeq2}) has solutions for infinitely many $n \in \N$.
    
    Now suppose $|\alpha| < 1$ and there are infinitely many $n \in \N$ such that Equation (\ref{eq: nosharefixeq2}) holds.
    Now, similarly, 
    since 
    \begin{equation}
        t^2_n - 4\alpha^n n \gamma\beta \frac{1- \alpha^n}{1- \alpha} = \frac{\gamma^2\beta^2}{(1- \alpha)^2}n^2 - 2\frac{\beta \gamma}{1 - \alpha}n + 1 - 2 \frac{\gamma^2 \beta^2}{(1- \alpha)^2}n^2 \alpha^n  - 2 \alpha^n + o(\alpha^n)
    \end{equation}
    when $n$ is large,
    we have 
    \begin{align}
    & X^\pm_n = \left(\frac{\gamma\beta}{1 - \alpha} n\alpha^n - \alpha^n - \frac{\beta \gamma}{1- \alpha} n + 1 \pm \left( \frac{\beta \gamma}{1 -\alpha}n - 1 - \frac{\gamma\beta}{1- \alpha}n\alpha^n - \alpha^n + o(\alpha^n)\right)\right) \nonumber \\ &\times (2n\alpha^n\gamma)^{-1} \nonumber
    \end{align}
    So,
    $$X^-_n =  -\frac{\beta}{ 1- \alpha}\alpha^{-n} + \frac{1}{\gamma n\alpha^n} + o(n^{-1}\alpha^{-n}),$$
    and we have
    $$ A(X^-_n)/B(X^-_n) \sim \alpha^{ln}$$
    for some integer $l$ when $n$ is large. However,
    $$ \alpha^nX^-_n + \frac{1 - \alpha^n}{1- \alpha}\beta \sim \frac{1}{n}$$
    contradicting that Equation (\ref{eq: nosharefixeq2}) has solution for infinitely many $n \in \N$. 

    Suppose $|\alpha| < 1$ and there are infinitely many $n \in \N$ such that Equation (\ref{eq: nosharefixeq1}) holds. 
    Then $X^+_n \sim \frac{1}{n }$
    when $n$ is large. 
    Thus, 
    $$ A(X^+_n)/B(X^+_n)\sim n^l$$
    for some integer $l$ and 
    $$ \alpha^n X^+_n + \frac{1 - \alpha^n}{1-  \alpha}\beta = \frac{\beta}{1 - \alpha}  - \frac{\beta}{1 - \alpha}\alpha^n + o(\alpha^n)$$
    when $n$ is large.
    If $l \neq 0$, then these contradict the assumption that the Equation (\ref{eq: nosharefixeq1}) has solutions for infinitely many $n \in \N$. If $l = 0$, then 
    $$ A(X^+_n)/B(X^+_n)  = A(0)/B(0) + O(1/n^m) $$
    for some positive integer $m$. Still, these contradicts the assumption that the Equation (\ref{eq: nosharefixeq1}) has solution for infinitely many $n \in \N$ by comparing the rates they approach constants when $n$ is large.\\

    Now, suppose $|\alpha| = 1$ but not a root of unity and there are infinitely many $n \in \N$ such that Equation (\ref{eq: nosharefixeq2}) holds. After eliminating denominators, combining terms and taking squares, we found that this is equivalent to that there exists a polynomial $F(Z, W)$ in two variables such that $F(n, \alpha^n) = 0$ for infinitely many $n \in \N$. By \cite[Proposition 2.5.1.4]{BGT}, we have $\{a_n = F(n, \alpha^n)\}_{n \in \N}$ satisfies a linear recurrence. Then, by \cite[Theorem 2.5.4.1]{BGT}, the set $\{n \in \N : a_n = 0\}$ is a finite union of arithmetic progressions and, since we assumed this set is infinite, we have that there exists $s_1$ and $s_2 \in \N$ such that $a_n = 0$ for every $n \in \{s_1 m + s_2 : m \in \N\}$. This is equivalent to saying that Equation (\ref{eq: nosharefixeq2}) holds for any $n \in \{s_1 m + s_2 : m \in \N\}$. Then, we consider $n$ within an infinite subsequence $I$ of $\{s_1 m + s_2 : m \in \N\}$ such that $|1 - \alpha^n|$ is bounded from below independently from $n$ and not converging. 
    
    Notice that we can always find such a subsequence $I$: since $\alpha^{n}$, $n \in R = \{s_1 m + s_2 : m \in \N \}$, is dense on the unit circle as $\alpha$ is not a root of unity, we can take a small $\epsilon >0$ independent of $n$ and then there are infinitely many $n \in R$ such that $\alpha^n \not \in B(1, \epsilon)$. Denote the subset of $R$ satisfies this condition as $I'$. Since, $R$ is dense in the unit circle $S^1$, we have $I'$ is also dense in $S^1 \setminus B(1,\epsilon)$. Thus, we can take a subsequence $I \subseteq I'$, such that $\alpha^n$ won't converge to any point on the unit circle when $n$ ranges in $I$. For example, we can achieve this by choosing two different points $p_1$ and $p_2$ on $S^1 \setminus B(1, \epsilon)$ and for each $r \in \N$, we pick a $n_{1,r}$ and $n_{2,r} \in I'$ such that 
    $$ |\alpha^{n_{1,r}} - p_1| < 1/r,$$
    $$ |\alpha^{n_{2,r}} - p_2| < 1/r,$$
    and let $I = \bigcup_{r \in \N}\{n_{1,r}, n_{2,r} \} \cap \{n \in \N : \alpha^n \in S^1 \setminus B(1, \epsilon)\}.$

    Now, we expand $X^{\pm}_n$ when $n$ goes large as follows:
    \begin{equation}
        X^{\pm}_n = \frac{-t_n \pm \sqrt{t^2_n - 4\alpha^n n \gamma (1 - \alpha^n)\beta/(1-\alpha)}}{2n\alpha^n \gamma},
    \end{equation}
    \begin{align}
        = \frac{-t_n  \pm \sqrt{ h_n + O(1)}}{2n\alpha^n \gamma} \nonumber,
    \end{align}
    where $$h_n = \frac{n^2 \gamma^2 (1 - \alpha^n)^2\beta^2}{ (1- \alpha)^2} + 2\frac{(\alpha^n-1)n\gamma(1 - \alpha^n)\beta}{(1-\alpha)}- 4\frac{\alpha^n n \gamma(1-\alpha^n)\beta}{(1-\alpha)}.$$
    Breaking the square root, we have
    \begin{equation}\label{eq: X_nwhenk=0}
        X^\pm_n = \frac{-t_n \pm (n \gamma(1-\alpha^n)\beta/(1- \alpha) -\alpha^n -1 + o(1))}{2n\alpha^n\gamma}.
    \end{equation}
    Thus, we have 
    \begin{equation}\label{eq: X+asyk=0}
        X^+_n = - \frac{1}{n\gamma} + o\left(\frac{1}{n}\right),
    \end{equation}
    when $n \in I$ becomes large. Then Equation (\ref{eq: nosharefixeq1}) implies 
    \begin{equation}
        c\left(-\frac{1}{n\gamma} + o\left(\frac{1}{n}\right)\right) = \alpha^n\left(- \frac{1}{n\gamma} + o\left(\frac{1}{n}\right)\right) + \frac{1 - \alpha^n}{1 - \alpha}\beta.
    \end{equation}
    However,
    \begin{equation}
        c\left(-\frac{1}{n\gamma} + o\left(\frac{1}{n}\right)\right) \sim (1/n)^l
    \end{equation}
    for some integer $l$ and 
    \begin{equation}
         \alpha^n\left(- \frac{1}{n\gamma} + o\left(\frac{1}{n}\right)\right) + \frac{1 - \alpha^n}{1 - \alpha}\beta =  (1 - \alpha^n)\frac{\beta}{1 - \alpha} + o(1)
    \end{equation}
    when $n$ is large. This gives a contradiction if $l \neq 0$. 
    
    Notice that if $l =0$, then we still have
    $$\lim_{n \to \infty}X^+_n = 0,$$
    by Equation (\ref{eq: X+asyk=0}).
    Then Equation (\ref{eq: nosharefixeq1}) implies that 
    \begin{equation}
        \lim_{n \to \infty} A(X^+_n)  = \lim_{n \to \infty} B(X^+_n) \left(\alpha^n \left(-\frac{1}{n \gamma} + o\left(\frac{1}{n}\right)\right) + \frac{1- \alpha^n}{1- \alpha}\beta\right).
    \end{equation}
    Now, if $\lim_{n \to \infty} B(X^+_n) = 0$, then this can only hold if $\lim_{n \to \infty} A(X^+_n) = 0$ which is impossible as $A(X)$ and $B(X)$ don't share common zeros. On the other hand, if $\lim_{n \to \infty} B(X^+_n) = b$, where $b$ is a non-zero constant, then 
    \begin{equation}
        \lim_{n \to \infty} A(X^+_n) = b \lim_{n \to \infty} \frac{ 1- \alpha^n}{1- \alpha}\beta,
    \end{equation}
    which is not converging. This is a contradiction as $\lim_{n \to \infty} A(X^+_n) = A(0)$.

    Now, suppose $|\alpha| = 1$ but not a root of unity and there are infinitely many $n \in \N$ such that Equation (\ref{eq: nosharefixeq1}) holds. Similarly as above, there exists $s_3, s_4 \in \N$ such that  Equation (\ref{eq: nosharefixeq1}) holds for any $n \in \{s_3 m + s_4| m \in \N\}$. We consider $n$ within an infinite subsequence of $I \subseteq \{s_3 m + s_4| m \in \N\}$ such that there exists a small $\epsilon > 0$ so that for all $n \in I$, we have $|1 - \alpha^n| > \epsilon$ and $\{\alpha^n : n \in I\}$ is dense in $S^1 \setminus B(1,\epsilon)$. This can be achieved in the same way as in the previous case since $\alpha^n$ is dense on $S^1$, when $n$ ranges in $ \{s_3 m + s_4: m\in \N\}$. Then, again by Equation (\ref{eq: X_nwhenk=0}), we have 
    \begin{equation}
        X^-_n = -\frac{\beta}{1- \alpha}\frac{1 - \alpha^n}{\alpha^n} + \frac{1}{n\alpha^n \gamma} + o\left(\frac{1}{n}\right).
    \end{equation}
    Therefore, Equation (\ref{eq: nosharefixeq2}) implies that 
    \begin{equation}
        c(X^-_n) = \frac{1}{n\gamma}+ o\left(\frac{1}{n}\right).
    \end{equation}
    However, for a non-zero rational function $c(X)$,
    \begin{equation}
       c\left(-\frac{\beta}{1- \alpha}\frac{1 - \alpha^n}{\alpha^n} + \frac{1}{n\alpha^n \gamma} + o\left(\frac{1}{n}\right)\right) 
    \end{equation}
    cannot converge to $0$ when $n$ becomes large since otherwise, as $1/\alpha^n$ is dense on $S^1 \setminus B(1,\epsilon)$, we would have 
    $$ c\left(- \frac{\beta}{1- \alpha} u + \frac{\beta}{1- \alpha}\right) = 0,$$
    for all $u \in S^1\setminus B(1,\epsilon)$ which implies that $c$ is constantly zero. This gives a contradiction.
   
\end{proof}

Now we can prove Theorem \ref{thm: mainlinear} and we restate it here for the convenience of readers:
\begin{thm} \label{thm: mainaut}
    Let $f(X)$, $g(X)$ be automorphisms on $\P^1$ defined over $\C$ and $c(X)$ be a rational function defined over $\C$. If the semigroup generated by $f(X)$ and $g(X)$ under compositions is free, then there are only finitely many $\lambda \in \C$ such that 
    \begin{equation} \label{eq: targeteq1}
         f^n(\lambda) = g^n(\lambda) = c(\lambda) 
    \end{equation}
    for some positive integer $n$ unless $f(X)$, $g(X)$ are simultaneously conjugated by an automorphism on $\P^1(\C)$ to one of the following:
    \begin{enumerate}
        \item $\alpha X + \beta$, $X/(\gamma X + \delta)$, with some $\alpha, \delta, \gamma, \beta \in \C^*$ such that one of $\alpha /\delta $ and $\alpha \delta$ is a root of unity;
        \item $\alpha X + \beta$, $\delta X + \gamma$,  with some $\alpha, \delta \in \C^*$ such that $\alpha$ and $\delta$ are not roots of unity, $\gamma$ and $\beta$ are not both $0$, and either $\alpha /\delta $ is a root of unity other than $1$ or one of $\alpha^2/\delta$ and $\delta^2/\alpha$ is a root of unity.
    \end{enumerate}
    
\end{thm}
Before we start the proof, we gather here four counter-examples in Remark (\ref{rmk: counter-1}), (\ref{rmk: counter-2}), (\ref{rmk: counter-3}), (\ref{rmk: counter-4}) and (\ref{rmk: counter-5}) corresponding to the two exceptional cases mentioned in the statement of Theorem \ref{thm: mainaut}. Therefore, Theorem \ref{thm: mainaut} is the best we can expect.

\begin{rmk}\label{rmk: counter-1}
    Suppose $f(X) = X + \beta$, $g(X) = X/(\gamma X + 1)$ and $c(X) = -\beta / (\gamma X)$, with some $\beta ,\gamma \in \C^*$, we have that there are infinitely many $n \in \N$ such that Equation (\ref{eq: targeteq1}) has solutions. One of the solutions is given by $$x_n = \left(-n\beta \gamma + \sqrt{n^2\beta^2 \gamma^2 - 4\beta \gamma}\right)/(2\gamma).$$  However, $f(X)$ and $g(X)$ may still generate a free semigroup. For example, when $f(X) = X + 2$ and $g(X) = X/(2X + 1)$, we denote 
    $$V_1 = \{ x \in \R^+ : x > 1\},$$
    $$ V_2 = \{ x \in \R^+ : x < 1\} .$$
    Then we have 
    $$ f(V_1) \subseteq V_1,$$
    $$ f(V_2) \subseteq V_1,$$
    $$ g(V_2) \subseteq V_2,$$
    $$ g(V_1) \subseteq V_2.$$
    Thus, the Ping-pong Lemma \ref{lem: ping-pong} implies $\langle f, g \rangle$ is a free semigroup under composition. 
\end{rmk}
\begin{rmk}\label{rmk: counter-2}
    Suppose $f(X) = \alpha X + \beta$ and $g(X) = X / (\gamma X + \alpha)$, where $\gamma$ and $\beta$ are non-zero and $\beta/(1 - \alpha) \neq (1- \alpha)/\gamma$. Denote 
    \begin{align*} K_1 &= \frac{\beta}{1 - \alpha} + \frac{1 - \alpha}{\gamma},\\
     K_2 &= \frac{\beta}{1 - \alpha} - \frac{1 - \alpha}{\gamma},\\
    K_3 &= \frac{\beta}{\gamma}.\end{align*}
    Let
    $$c(X) = \frac{K_2 X^2}{-2K_3 -2 X^2 + K_1X} + \frac{\beta}{1 - \alpha} - \frac{\beta}{1 - \alpha} \frac{K_2 X}{-2K_3 -2 X^2 + K_1X}.$$
    One can directly verify that 
    $$ c(X) = f^n(X) = g^n(X)$$
    has 
    $$ X_n = \frac{ - K_1 + K_2 \alpha^{-n}  - \sqrt{(- K_1 + K_2 \alpha^{-n})^2 - 4}}{2}$$
    as a solution for each $n \in \N$. Furthermore, suppose $\gamma = 1$, $\beta =1$ and $\alpha = 2$ for example. Then, similarly as in Remark (\ref{rmk: counter-1}), one can use the Ping-pong Lemma \ref{lem: ping-pong} to show that $\langle f(X), g(X) \rangle$ is a free semigroup.
\end{rmk}
\begin{rmk}\label{rmk: counter-3}
    Suppose $f(X) = \alpha X + \beta$ and $g(X) = \delta X +\gamma$, where $\alpha$ and $\delta$ are not roots of unity but $\xi = \delta/\alpha$ is a root of unity of order $l > 1$. For each $i \in \{1,2, \dots, l-1\}$, denote 
    $$ K_{1,i} = \frac{1}{1 - \xi^i} \left(\frac{\gamma}{1- \delta} - \frac{\beta}{ 1- \alpha}\right),$$
$$ K_{2,i} = \frac{1}{1 - \xi^i} \left(\frac{\xi^i \gamma}{1- \delta} - \frac{\beta}{ 1- \alpha}\right). $$
Let $$c_i(X) = K_{1,i} + \frac{\beta}{1 - \alpha} - \frac{K_{1,i}}{X + K_{2,i}} \left(K_{2,i} + \frac{\beta}{1- \alpha}\right).$$
We have that 
$$ X_{n,i}= K_{1,i}\alpha^{-nl-i} - K_{2,i}$$
is a solution to 
$$ f^{nl+i}(X) = g^{nl+i}(X) = c_i(X)$$
for every $n \in \N$ and $i \in \{1,2, \dots, l-1\}. $
For an explicit example, if we take $f(X) = 2X +1$ and $g(X) = -2X$ then $\langle f, g \rangle$ is a free semigroup by \cite[Theorem 2]{KT21} as $f \circ g \neq g \circ f$. So this provides a counter-example such that the finiteness of solutions to $f^n(X) = g^n(X) = c(X)$ dooesn't hold even when $f$ and $g$ generate a free semigroup.
\end{rmk}
\begin{rmk}\label{rmk: counter-4}
    Suppose $f(X) = \alpha X + \beta $ and $g(X)= X/(\gamma X + \xi \alpha^{-1})$ where $\xi$ is a root of unity of order $m$, $\alpha \in \C^*$ is not a root of unity, $\gamma$ and $\beta$ are non-zero in $\C$. 
    Let's denote 
   $$
        K_1 \vcentcolon = \frac{\beta}{1 - \alpha},
    $$
    $$
    K_2 \vcentcolon = \frac{\gamma}{1 - \xi \alpha^{-1}}.
    $$
    Let's also define 
    $$
    c(X) = \frac{K_1K_2X - K_1}{K_1K_2 - K_2X} + K_1\left(1 - \frac{-K_1 + K_1K_2X}{K_2X(K_1-X)}\right).
    $$
    Then one can check directly that 
    \begin{align}
       & X_{n} = \left( - K_1K_2(1 - \alpha^n)(1 - \alpha^{-n})  \right. \nonumber \\ 
       & \left. -\sqrt{K^2_1K^2_2(1 - \alpha^n)^2(1 - \alpha^{-n})^2 - 4K_1K_2 (1 -\alpha^n)(1 - \alpha^{-n})} \right) \nonumber \\ \times &\left( 2\alpha^n \frac{1 - \alpha^{-n}}{1 - \xi \alpha^{-1}} \gamma\right)^{-1}
    \end{align}
    when $n \in \{km : k \in \N\}$ is an infinite sequence satisfying 
    \begin{equation}
        f^n(X_n) = g^n(X_n) = c(X_n).
    \end{equation}
     
    Since the direct computation might be very complicated, we will demonstrate here on how to construct such a $c(X)$ to make this hold. Let $\delta = \xi \alpha^{-1}$. Then the sequence of equations:
    \begin{equation}
        f^n(X_n) = g^n(X_n),
    \end{equation}
     when $n \in \{km : k \in \N\}$, will give a sequence of quadratic equations with variables $X_n$'s:
    $$
    \alpha^n(1 - \alpha^{-n})K_2 X^2_n + K_1K_2(1 - \alpha^n)(1 - \alpha^{-n})X_n + K_1(1 - \alpha^{n})\alpha^{-n} = 0 
    $$
    and we will denote the two solutions of these quadratic equations as $X^-_n$ and $X^+_n$. Then we have 
    \begin{equation}
        X^-_n X^+_n = \frac{\alpha^{-n}(1 - \alpha^n)}{\alpha^n(1 - \alpha^{-n})} \frac{K_1}{K_2} = - \alpha^{-n} \frac{K_1}{K_2},
    \end{equation}
    \begin{equation}
        X^-_n + X^+_n = \frac{ - K_1K_2(1 - \alpha^n)(1 - \alpha^{-n})}{ \alpha^n K_2(1 - \alpha^{-n})} = K_1 (1 - \alpha^{-n}),
    \end{equation}
    where $n \in \{km : k \in \N\}$.
    Combining these two equations, we get 
    \begin{equation}
        K_1(1 - \alpha^{-n})X^-_n - (X^-_n)^2 = - \alpha^{-n}\frac{K_1}{K_2},
    \end{equation}
    which implies
    \begin{equation}\label{eq: alpha^neq}
        \alpha^n = \frac{- K_1 + K_1K_2X^-_n}{X^-_nK_2(K_1 - X^-_n)}.
    \end{equation}
    Let $X_n = X^-_n$. Then using that 
    $$c(X_n) = f^n(X_n) = \alpha^n X_n + \frac{1 - \alpha^n}{1 - \alpha}\beta$$ 
    and plugging in Equation (\ref{eq: alpha^neq}) will give us the desired $c(X_n)$.

    Furthermore, notice that when $\alpha = 2$, $\beta =\gamma = \xi =  1$, we can use the Ping-pong lemma, Lemma \ref{lem: ping-pong}, similarly as in Remark \ref{rmk: counter-1} by choosing 
     $$V_1 = \{ x \in \R^+ : x > 1\},$$
    $$ V_2 = \{ x \in \R^+ : x < 1\} ,$$
    to argue that $\langle f(X), g(X) \rangle$ is a free semigroup under the compositions.

\end{rmk}
\begin{rmk}\label{rmk: counter-5}
    Let $\alpha$ be any non-zero and not roots of unity complex number. Let $\mu$ be a root of unity of order $m \in \N$. Let $\delta = \mu \alpha^2$ and $\beta = \alpha -1$. Then we have $X_{mn} = \alpha^{-mn}$ satisfies
    \begin{equation}
        f^{mn}(X_{mn}) = g^{mn}(X_{mn}) = 1/X_{mn}
    \end{equation}
    for all $n \in \N$, where $f(x) = \alpha x + \beta$ and $g(x) = \delta x $.

    Furthermore, when $\alpha = 2$, $\delta = 4$ and $\beta = 1$, let $V_1$ denote the set of odd natural numbers and $V_2$ denote the set of even natural numbers. We have 
    $$ f(V_1 \cup V_2) \subseteq V_1,$$
    $$ g(V_1 \cup V_2) \subseteq V_2.$$
    Thus, Lemma \ref{lem: ping-pong} tells us that $\langle f, g \rangle$ is a free semigoup under the compositions.
    
\end{rmk}
\begin{proof}[Proof of Theorem \ref{thm: mainaut}]
    
   After conjugation, we assume that 
    $$ f(X) = \alpha X + \beta, $$
    $$ g(X) = \frac{X}{\gamma X + \delta}.$$
    We always assume that $\alpha$ and $\delta$ are non-zero as otherwise it will result in a constant map. \newline

    Suppose that at least one of $\alpha$ and $\delta$ is a root of unity other than $1$. Assume without loss of generality it is $\delta$. Then there exists a positive integer $m$ such that $\delta^m = 1$ and
    $$g^m(X) = \frac{X}{\delta^m} = X.$$
    Therefore, $\langle g^m(X), f(X) \rangle$ is not a free semigroup under composition and there is nothing to show.
    \newline

Now, we assume that neither $\alpha$ nor $\delta$ is a root of unity other than $1$ and $\alpha/ \delta$, $\alpha \delta$ are not a root of unity. By Lemma \ref{lem: reduce-to-sharingfixpoint}, we can assume that $f$ and $g$ share at least one common fixed point.

    If $f$ and $g$ share two fixed points, then $\beta = \gamma = 0$ and they are both scaling maps and certainly not generating a free semigroup under composition.
    
    Suppose $f$ and $g$ share only one common fixed point. Now we first assume that $c(X)$, $f(X)$ and $g(X)$ are all defined over $\overline{\Q}^*$. Then after conjugating $c(X)$, $f(X)$ and $g(X)$ by a common automorphism, we assume that $f(X) = \alpha X + \beta $ and $g(X) = \delta X + \gamma$, where $\alpha, \delta \in \overline{\Q}^*$, $\beta, \gamma \in \overline{\Q}$ and $\beta, \gamma$ are not both $0$. Now our discussion is divided into several cases:
    \newline

    \noindent\textbf{Case I.} Suppose at least one of $\delta$ and $\alpha$ is a root of unity other than $1$. Without loss of generality we assume it is $\delta$. Then there exists $m$ such that $g^m(X) = X$ and we have $g^m \circ f = f \circ g^m$. Thus, $\langle f ,g \rangle$ is not a free semigroup. Same argument works if $\alpha$ is a root of unity other than $1$ and $\delta$ is not a root of unity.
    \newline 

    \noindent\textbf{Case II.} Suppose $\delta = 1$ and $\alpha$ is not a root of unity. 
    \begin{lem}\label{lem: thmCase2}
       There are only finitely many $n \in \N$ such that $c(X) = f^n(X) = g^n(X)$ has solution.
    \end{lem}
    This concludes this case.
    \newline

    \noindent\textbf{Case III.} Suppose both $\alpha$ and $\delta$ are not roots of unity and $\alpha/\delta$, $\alpha^2 /\delta$ and $\delta^2/\alpha$ are not roots of unity. Without loss of generality, we assume that $\beta \neq 0$ since $\gamma$ and $\beta$ are not both $0$ and we can always relabel $f$ and $g$ if necessary. We have $$X_n = \left( \frac{1 - \delta^n}{1 -\delta}\gamma - \frac{1 - \alpha^n}{1 - \alpha} \beta\right)/(\alpha^n - \delta^n)$$ is the solution to 
    $$ f^n(X) = g^n(X).$$
    If there are infinitely many $n \in \N$ such that the Equation (\ref{eq: targeteq1}) has solutions, then 
    $$ c(X_n) = f^n(X_n)$$
    for infinitely many $n \in \N$. 
    It means that 
    $$ A(X_n) = \alpha^n B(X_n)X_n + \frac{1 - \alpha^n}{1 - \alpha}\beta B(X_n)$$
    holds for infinitely many $n \in \N$.
    This is equivalent to say 
    \begin{equation}\label{eq: polynomialformal1}
        F(Z,W) = A(X(Z,W)) - ZB(X(Z,W))X(Z,W)-\frac{1 - Z}{1 - \alpha}\beta B(X(Z,W)) = 0
    \end{equation}
    holds for infinitely many $\{(\alpha^n , \delta^n) : n \in \N\}$, where 
    $$ X(Z,W) = \left( \frac{1 - W}{1 -\delta}\gamma - \frac{1 - Z}{1 - \alpha} \beta\right)\Big/(Z - W).$$
    
    We first claim that the rational function $F(Z,W)$ is not constantly zero. To prove this, we suppose on the contrary that Equation (\ref{eq: polynomialformal1}) holds constantly for any choice of values of $Z$ and $W$.
    Let's denote $m_1 = \deg(A)$ and $m_2 = \deg(B)$.
    
    Then choose a specialization $Z = W+1$, we have $$X(Z,Z-1) = \left(\frac{\beta}{1 - \alpha} - \frac{\gamma}{ 1- \delta}\right)Z + \frac{2\gamma}{1-\delta} - \frac{\beta}{1-\alpha}$$ in this case. Then Equation (\ref{eq: polynomialformal1}) holds constantly implies that $m_1 = m_2 + 2$ unless
    \begin{equation}\label{eq: sharingtwofix1}
        \frac{\beta}{1 - \alpha} =  \frac{\gamma}{1-\delta}.
    \end{equation} While Equation (\ref{eq: sharingtwofix1}) implies $f(X)$ and $g(X)$ share two fixed points and this case has already been handled. So, we assume Equation (\ref{eq: sharingtwofix1}) doesn't hold and thus $m_1  = m_2 + 2$. 
    
    On the other hand, we can choose another specialization $ Z = 0$.
    Then 
    $$ X(0, W) = \frac{\gamma}{1 - \delta} + \left( \frac{\beta}{1 - \alpha} - \frac{\gamma}{1 - \delta}\right)W^{-1},$$
    and Equation (\ref{eq: polynomialformal1}) implies that 
    $$A(X(0,W)) = \frac{\beta}{1 - \alpha} B(X(0,W)).$$
    This holds constantly for all $W$ will require that $m_1 = m_2$, which gives a contradiction.
    Therefore, Equation (\ref{eq: polynomialformal1}) doesn't hold constantly for every values of $Z$ and $W$.

    Let's then write 
    $$ F(Z,W) = \sum_{0 \leq i  \leq k_1, 0 \leq j \leq k_2} c_{i,j}Z^iW^j,$$
    for some positive integers $k_1$ and $k_2$ after clearing the denominators. Then we use the same argument as in Lemma \ref{lem: reduce-to-sharingfixpoint} by applying S-unit Theorem to conclude that $\alpha$ and $\delta$ are multiplicatively dependent. Thus, there exists $k \in \Q^*$ such that $\xi \alpha^k= \delta$, for some root of unity $\xi$ of order $k_3 \in \N$. Note that, we can further relabel $f$ and $g$ if necessary (we will not assume $\beta \neq 0$ from now on which is an assumption obtained by a specific choice of labeling) so that $k \geq  1$ or $k \leq -1$.

    Therefore, we have 
    $$ X_{n} = \left(\frac{1 - \xi^n\alpha^{nk} }{1- \xi\alpha^k}\gamma - \frac{1 - \alpha^n}{1- \alpha}\beta\right)\Big/(\alpha^n - \xi^n\alpha^{nk}).$$
    Then following the same idea as how we argue in the proof of Lemma \ref{lem: reduce-to-sharingfixpoint} (Case \textbf{\RNum{2}}, Step \textbf{\RNum{3}}), it is enough to show the claim that 
    \begin{equation}\label{eq: polynomialcase4eq3}
         A(X_i(Z)) = \left(ZX_i(Z) + \frac{1 - Z}{1 - \alpha}\beta\right)B(X_i(Z))
    \end{equation}
    
    doesn't hold constantly for any $i \in \{1,2, \dots, k_3  \}$ unless $f$ and $g$ share two common fixed points, where 
    $$ X_i(Z) =\left(\frac{1 - \xi^i Z^{k} }{1- \xi \alpha^k}\gamma - \frac{1 - Z}{1- \alpha}\beta\right)\Big/(Z - \xi^i Z^{k}). $$
    This is because, similarly as in Lemma \ref{lem: reduce-to-sharingfixpoint}, Equation (\ref{eq: polynomialcase4eq3}) is an equation on one variable and having infinitely many roots implies it is constantly zero. 
    
    To prove the claim, we can assume that $k \neq 0$ since otherwise $\delta$ is a root of unity and we are back to the case that we have discussed. We first suppose $k > 0$ and suppose the contrary that there exists a $$i \in \{1,2, \dots, k_3  \} $$ such that Equation (\ref{eq: polynomialcase4eq3}) holds constantly. Notice that in the beginning of Case \textbf{\RNum{3}} we assumed that $\alpha/\delta$, $\alpha^2/\delta$ and $\delta^2/\alpha$ are not roots of unity, which is equivalent to $k \neq   1,2$ as we assumed $k \geq 1$.
    Also, if $f$ and $g$ only share one common fixed point, then $$\frac{\gamma}{1-\delta} \neq \frac{\beta}{1 - \alpha}$$ and
    $$\val_Z(X_i(Z)) = -  \min\{1, k\}= -1,$$
    $$\val_Z(A(X_i(Z))) = -m_1 \min\{1, k\} = -m_1,$$
    $$ \val_Z\left(\left(ZX_i(Z) + \frac{1 - Z}{1 - \alpha}\beta\right)B(X_i(Z))\right) \geq  -m_2 \min\{1,k\}  = -m_2.$$
    Suppose, Equation (\ref{eq: polynomialcase4eq3}) holds constantly for every values of $Z$, then $$m_2 \geq m_1.$$ 

    Denote $k = l_1/l_2$ where $l_1$, $l_2$ are coprime positive integers and $l_1 > l_2$ and denote $l_3$ as the order of $\xi^{-i}$. Consider the change of variable by introducing $T$ such that $T^{l_2} = Z$. Then
    $$ X_i(T)  = \left ( \frac{1 - \xi^iT^{l_1}}{1- \delta} \gamma - \frac{1 - T^{l_2}}{1 -\alpha} \beta \right)/(T^{l_2} - \xi^i T^{l_1}),$$ and 
\begin{equation}\label{eq: alpha-2-delta-1}
    A(X_i(T)) = \left(T^{l_2}X_i(T) + \frac{1 - T^{l_2}}{1 -\alpha}\beta \right)B(X_i(T)).
\end{equation}

Since $k \neq 1, 2$, we have $l_2 \neq l_1 - l_2$. For a $(l_1 -l_2)l_3$-th root of unity $\xi_0$ such that $ \xi^{l_1 -l_2}_0 = \xi^{-i}$, we have that for any $\mu \in \C^*$ such that $\mu^{l_1 - l_2} = 1$ the root of unity $\mu \xi_0$ is a zero of $T^{l_2} - \xi^iT^{l_1}$. On the other hand, for any $\mu \in \C$ such that $\mu^{l_1 - l_2} = 1$,
\begin{equation}
    \frac{1 - \xi^i (\mu \xi_0)^{l_1}}{1 - \delta}\gamma - \frac{1-(\mu \xi_0)^{l_2}}{1- \alpha}\beta = \left (\frac{\gamma}{1 - \delta} - \frac{\delta}{1 -\alpha}\right)(1 - \mu^{l_2}\xi_0^{l_2}). 
\end{equation}
Note that $\gamma /(1 - \delta) \neq \beta /(1 - \alpha)$. If $l_1 - l_2 > 1$, then there exists a $\mu_0 \in \C$ such that $\mu^{l_1 -l_2}_0 = 1$ and $1 - \mu^{l_2}_0 \xi_0^{l_2} \neq 0$, since otherwise $l_2 = m (l_1 - l_2)$ for some $m \in \N^+$ and so $l_1 = (m+1)(l_1 - l_2)$, contradicting that $l_1$ and $l_2$ are coprime. Note that this implies $T_0 \coloneqq \mu_0\xi_0$ is a pole of $X_i(T)$.

But then, $ \val_{T - T_0}(X_i(T)) \leq -1$ and Equation (\ref{eq: alpha-2-delta-1}) implies that
\begin{align}
    \val_{T-T_0}(A(X_i(T))) = m_1\val_{T - T_0}(X_i(T)) \\= (1 + m_2) \val_{T - T_0}(X_i(T)) = \val_{T - T_0}\left(\left(T^{l_2}X_i(T) + \frac{1 - T^{l_2}}{1 -\alpha}\beta \right)B(X_i(T))\right), \nonumber
\end{align}
which further implies that $m_1 = m_2 + 1$. This contradicts the inequality $m_2 \geq m_1$ we obtained above. 

Now, Suppose $l_1 = l_2 +1$. If $\xi^{l_2i}\neq 1$, then $T_1 \coloneqq \xi^{-i}$ is a pole of $$X_i(T) =  \left (\frac{\gamma}{1 - \delta}(1 - \xi^i T^{l_2 + 1}) - \frac{\delta}{1 -\alpha}(1 - T^{l_2})\right)\left(T^{l_2}(1 - \xi^iT)\right).$$ Then the valuation with respect to $T - T_1$ again gives the contradiction to $m_1 \leq m_2$.

Now, we only left to show the case when $l_1 = l_2 + 1$ and also $\xi^{il_2} = 1$. Note that we only need to obtain a contradiction when $l_2 > 1$, since we assumed $k \neq 2$. In this case, we can rewrite $X_i(T)$ as follows 
\begin{align}
    X_i(T) = \left (\frac{\gamma}{1 - \delta}(1 - (\xi^i T)^{l_2 + 1}) - \frac{\beta}{1 -\alpha}(1 - (\xi^iT)^{l_2})\right)\left(T^{l_2}(1 - \xi^iT)\right) \nonumber \\
    = \left(\left( 1 + \xi^iT + \dots + (\xi^iT)^{l_2-1}\right)\left(\frac{\gamma}{1 -\delta} - \frac{\beta}{1 - \alpha} \right) + \frac{\gamma}{1 - \delta}T^{l_2} \right)/T^{l_2}. 
\end{align}
When $T$ is close to $0$, we have 
\begin{equation}
    X_i(T) = L (T^{-l_2} + \xi^iT^{1 - l_2}) + o(T^{1 - l_2}),
\end{equation}
where we denote $L \coloneqq \gamma/(1 - \delta) - \beta/(1 - \alpha) $.
Then, we similarly have 
\begin{equation}
    T^{l_2}X_i(T) + \frac{1 - T^{l_2}}{1 -\alpha}\beta = \frac{\gamma}{1 - \delta} + L \xi^iT + o(T).
\end{equation}

Now, by looking at the $\val_{T}$, if $\gamma  = 0$, then Equation (\ref{eq: alpha-2-delta-1}) implies that $-\deg(A)l_2 = -\deg(B)l_2 + 1$, which is impossible since $l_2 > 1$. If $\gamma \neq 0$, then by looking at the $\val_T$ of Equation (\ref{eq: alpha-2-delta-1}), we have $\deg(A) = \deg(B)$. Denote $m \coloneqq \deg(A)$, $a_m$ and $b_m$ the leading coefficients of $A$ and $B$. Then  
\begin{align}
    A(X_i(T)) = a_m L^m T^{-ml_2} + ma_mL^m \xi^iT^{1 - ml_2} + o(T^{1 - ml_2}),
\end{align}
\begin{equation}
    B(X_i(T)) = b_m L^m T^{-ml_2} + mb_mL^m \xi^iT^{1 - ml_2} + o(T^{1 - ml_2}).
\end{equation}
Then Equation (\ref{eq: alpha-2-delta-1}) implies that 
\begin{align}
    a_m L^m T^{-ml_2} + ma_mL^m \xi^iT^{1 - ml_2} + o(T^{1 - ml_2}) \nonumber \\ = \left( \frac{\gamma}{1 - \delta} + L \xi^iT + o(T)\right) \left(b_m L^m T^{-ml_2} + mb_mL^m \xi^iT^{1 - ml_2} + o(T^{1 - ml_2})\right).
\end{align}
Comparing the coefficients of the first two leading terms, we have
\begin{equation}
    a_m = \frac{\gamma}{1 -\delta}b_m,
\end{equation}
\begin{equation}
    \frac{\gamma}{1 - \delta} mb_mL^m + b_mL^{m+1} = ma_mL^m.
\end{equation}
Since $b_m \neq 0$, these together imply that $L=0$, which implies that $f$ and $g$ share two common fixed points. Hence $\langle f, g \rangle$ is not a free semigroup under the compositions.

Thus, the claim is proved for $k > 0$. \\


Now suppose $k < 0$. Then, it means that there exists a pair of positive integers $(l_1, l_2)$ such that $\alpha^{l_1} \delta^{l_2} = 1$. Thus we can check directly 
\begin{equation}
    f^{l_1}(X) = \alpha^{l_1}X + \beta',
\end{equation}
\begin{equation}
    g^{l_2}(X) = \delta^{l_2}X + \gamma',
\end{equation}
\begin{equation}
    f^{l_1} \circ g^{l_2} = X + c_1,
\end{equation}
\begin{equation}
    g^{l_2} \circ f^{l_1} = X + c_2,
\end{equation}
for some constants $c_1$, $c_2$, $\beta'$ and $\gamma'$. Then obviously $ f^{l_1} \circ g^{l_2}$ commutes with $g^{l_2} \circ f^{l_1} $, which implies that $f(X)$ and $ g(X)$ don't generate a free semigroup under composition. Thus, it contradicts the assumption. \newline
\newline 
 \noindent\textbf{Case \RNum{4}.} Lastly, suppose $\alpha = \delta$, then $f^n(X) = g^n(X)$ only has a solution if $f = g$ and certainly $\langle f , g \rangle$ is not free in this case.

Now, consider the general case that $f(X)$, $g(X)$ and $c(X)$ are all defined over $\C$. We use a similar specialization argument as in the proof of \cite[Theorem 2]{HT17}. We assume $\alpha$, $\delta$ and $\alpha/\delta$ are not roots of unity, $\gamma/(1 - \delta) - \beta/ (1- \alpha) \neq 0$ and $\beta$, $\gamma$ are not both zero, since otherwise the argument above directly shows that $\langle f, g \rangle$ is not free. We construct a ring extension $R$ over $\Z$ generated by coefficients of $f(X)$, $g(X)$ and $c(X)$. Denote $K = \Frac(R)$. We do induction on the transcendence degree, $N$, of $K$. The base case is that $N=0$, which means everything is algebraic and then we showed from the discussion above that Equation (\ref{eq: targeteq1}) can only have solution for finitely many $n \in \N$. 

Now, assume that if the transcendence degree is $N - 1$, then the above assumptions on $\gamma$ , $\beta$, $\alpha$, $\delta$, $\alpha/\delta$ and $\gamma/(1 - \delta) - \beta/ (1- \alpha) $ implies that the equation only admits solutions for finitely many $n \in \N$. Take $L$ as a subfield of $K$ of transcendence degree $N-1$. We will use the notation $\hat{h}_{\overline{L}, f}$ to denote the canonical height function associated to the function $f$ defined over $\overline{L}$. Applying \cite[Theorem 4.1]{CS93}, we have that for all specializations $s : R \to \overline{L}$ of sufficiently large height, we have $\hat{h}_{\overline{L},x^2}(\alpha_s)$, $\hat{h}_{\overline{L},x^2}(\delta_s)$ and $\hat{h}_{\overline{L},x^2}(\alpha_s/\delta_s)$ are all greater than $0$ and if $\beta$, $\gamma$ and $\gamma/(1 - \delta) - \beta/(1 - \alpha)$ are not roots of unity or zero, we also have $\hat{h}_{\overline{L},x^2}(\beta_s)$,$\hat{h}_{\overline{L},x^2}(\gamma_s)$ and $\hat{h}_{\overline{L},x^2}(\gamma_s/(1 - \delta_s) - \beta_s/(1 - \alpha_s))$ are non-zero. This, together with the assumption that 
$$ \gamma/(1 - \delta) - \beta/(1 - \alpha) \neq 0,$$
$\beta $, $\gamma$ are not both zero and $\alpha$, $\delta$, $\alpha/\delta$ are not roots of unity, implies, in particular, that none of $\alpha_s$, $\delta_s$ and $\alpha_s/\delta_s$ is a root of unity or zero, $\gamma_s/(1 - \delta_s) - \beta_s/(1 - \alpha_s)$ are nonzero and $\beta_s$ and $\gamma_s$ are not both zero. Then the induction hypothesis shows that 
$$ f^n_s(X) = g^n_s(X) = c_s(X)$$
has solution for only finitely many $n \in \N $ and thus 
$$f^n(X) = g^n(X) = c(X)$$
has solution for only finitely many $n \in \N$.

\end{proof}
\begin{proof}[Proof of Lemma \ref{lem: thmCase2}]
     We have that 
    $$ f^n(X) = \alpha^n X + \frac{1 - \alpha^n}{1 - \alpha}\beta$$
    $$ g^n(X) = X + n\gamma.$$
    If $\gamma = 0$, then $\langle f , g \rangle$ is obviously not free, so we assume that $\gamma \neq 0$.
    Then \begin{equation}\label{eq: X_n-growth}
        X_n = \frac{\beta}{ 1- \alpha}  - \frac{n\gamma}{1- \alpha^n}
    \end{equation} is the solution to 
    $$ f^n(X_n) = g^n(X_n).$$

   Recall the notation that $c(X) = A(X)/B(X)$, for some coprime polynomials $A(X)$ and $B(X)$. Now suppose the equation
    \begin{equation}\label{eq: polycase3eq1}
        A(X_n) = B(X_n)\left(\alpha^n X_n + \frac{1 - \alpha^n}{1 - \alpha}\beta\right)
    \end{equation}
    is satisfied by infinitely many $n \in \N$. If $|\alpha| > 1$, then 
    $$A(X_n)/B(X_n) \sim (n/\alpha^n)^l$$
    for some integer $l$ when $n$ is large. However, 
    $$ \left(\alpha^nX_n + \frac{1 - \alpha^n}{1- \alpha}\beta\right) \sim \gamma n$$
    when $n$ is large,
    which contradicts the assumption that the Equation (\ref{eq: polycase3eq1}) holds for infinitely many $n \in \N$.
    
    On the other hand, suppose $|\alpha|< 1$. Now, take a set of places $S$ such that for any $v \notin S$, we have $A(X)$ and $B(X)$ has coefficients with norm $1$ with respect to $v$ and also 
    $$ |1 - \alpha|_v = |\alpha|_v = |\beta|_v = |\gamma|_v = 1,$$
    if $\beta \neq 0$.
    Then, by the proof of \cite[Proposition 15, \noindent\textbf{Case II}]{HT17}, when $n$ is large enough, there is a $v$ outside of $ S$ such that 
    $$ | n/(1 - \alpha^n)|_v > 1.$$
    Thus, the Equation (\ref{eq: polycase3eq1}) implies that 
    \begin{equation}\label{eq: polycase3eq2}
        \deg(A) = \deg(B) +1
    \end{equation}
    since $|X_n|_v > 1$ with the choice of $n$ and $v$ such that $| n/(1 - \alpha^n)|_v > 1$. Now, since 
    $$ \lim_{n \to \infty} |\alpha^n X_n| = 0,$$
    we have 
    $$ \lim_{n \to \infty } \left(A(X_n) - \frac{\beta}{ 1- \alpha} B(X_n)\right) = 0$$
    by Equation (\ref{eq: polycase3eq1}) and the fact 
    $$ \lim_{n \to \infty} B(X_n)\alpha^nX_n = 0.$$
    Notice that if $\beta = 0$, we would have $A(X) \equiv 0$, which contradicts Equation (\ref{eq: polycase3eq2}). Now, if $\beta \neq 0$, this together with
    $$ \lim_{n \to \infty} |X_n| = \infty,$$
    imply that 
    $$\deg(A) = \deg(B)$$
    which contradicts  Equation (\ref{eq: polycase3eq2}).

    Now, suppose $|\alpha| = 1$. By the same argument from the last case, we have $\deg(A) = \deg(B) +1 $. Let 
    $$F(n, \alpha^n) =\left( A(X_n) - B(X_n)\left(\alpha^n X_n + \frac{1 - \alpha^n}{1 - \alpha}\beta\right)\right)(1 - \alpha^n)^{\deg(A)}$$ 
    be a polynomial in $n$ and $\alpha^n$. Then by \cite[Proposition 2.5.1.4]{BGT} we have $\{F(n, \alpha^n)\}_{n \in \N}$ satisfies a linear recurrence. Thus, by \cite[Theorem 2.5.4.1]{BGT}, the set $\{n \in \N : F(n , \alpha^n) = 0\}$ is a finite union of arithmetic progressions. Since we assumed that $F(n, \alpha^n) = 0$ holds for infinitely many $n$, there exists $s_1, s_2 \in \N$ such that $F(n, \alpha^n) = 0$ for any $n \in \{s_1 m + s_2 : m \in \N\}$. Now, we consider $n$ within an infinite subsequence $I$ of $\{s_1 m + s_2 : m \in \N\}$ such that $|1 - \alpha^n| > \epsilon $ for some $\epsilon > 0$ independent from $n$ and $\alpha^n$ is not converging. The existance of such an infinite subsequence $I$ is argued in the last part of the proof of Lemma \ref{lem: lemCase1}. Thus 
    $$\lim_{n \to \infty\text{, }n\in I} F(n, \alpha^n) = 0,$$
    where $n$ ranges in the chosen subsequence $I$. Notice that by Equation (\ref{eq: X_n-growth}), we have $$X_n = -\gamma n/(1 - \alpha^n) + o(n/(1- \alpha^n))$$ for $n \in I$ and $n$ becomes large. By looking at the leading terms of $F(n, \alpha^n)$ as $n$ grows, we obtain that 
    \begin{equation}
        \lim_{n \to \infty, n \in I} (a - b \alpha^n)\left(- \frac{\gamma}{1 - \alpha^n}\right)^{\deg(A)}n^{\deg(A)} + o(n^{\deg(A)}) = 0,
    \end{equation}
    where $a$ and $b$ are the leading coefficients of $A(X)$ and $B(X)$ respectively, which implies that
    $$ \lim_{n \to \infty\text{, }n\in I} \alpha^n = a/b,$$
    since $|1/(1 - \alpha^n)|$ is bounded from both below and above.
     This gives a contradiction as $\alpha^n$ doesn't converge.
\end{proof}
\section{Proof of Theorem \ref{thm : mainnonlinear}}

The following lemma asserts that an infinite sequence $\{x_n\}$ of solutions to the equation $f^n(x_n)=c(x_n)$ for all $n\in\mathbb{N}$ is a sequence of small dynamical height.
\begin{lem}\label{smallpoints} Let $K$ be a global field. Let $\{x_n\}$ be a sequence in $\mathbb{P}^1(\overline{K})$ so that $f^n(x_n)=c(x_n)$ for all $n\in\mathbb{N}$ where $f,c\in K(x)$ and $\mathrm{deg}(f)>1$. Then $$\lim_{n\rightarrow\infty}\hat{h}_f(x_n)=0.$$
\end{lem}
\begin{proof} For each $x\in \mathbb{P}^1(\overline{K}$), we have by Theorem \ref{prop: canoheightproperties} that
$$\hat{h}_f(c(x))=(\mathrm{deg}\, c)\hat{h}_f(x)+O(1)$$ where the implied constants do not depend on $x$ (but coefficients of $f$). For each $n\in\mathbb{N}$, the sequence $\{x_n\}_{n\in\mathbb{N}}$ satisfies $f^n(x_n)=c(x_n)$ and the functorial property of canonical height yields
$$\hat{h}_f(c(x_n))=\hat{h}_f(f^n(x_n))=(\mathrm{deg}\, f)^n\hat{h}_f(x_n).$$ This implies
$$(\mathrm{deg}\, f)^n\hat{h}_f(x_n)=(\mathrm{deg}\, c)\hat{h}_f(x_n)+O(1)$$ where the implied constant does not depend on $x_n$ and so does not depend on $n$. Thus the quantity $((\mathrm{deg}\, f)^n-\mathrm{deg}\,c)\hat{h}_f(x_n)$ is bounded by a constant independent of $n$. As the term $((\mathrm{deg}\, f)^n-\mathrm{deg}\,c)$  becomes arbitrary large as $n$ gets large because $\mathrm{deg}\,f>1$, it follows that $\hat{h}_f(x_n)$ must become arbitrary small and approaches to $0$ as $n\rightarrow+\infty$.
\end{proof}

The following result is well-known in arithmetic dynamics and it asserts that if two rational functions share the same (non-repeating) sequence of points of small height, their dynamical height must be identical.    The strategy of the proof is similar to that of Baker-DeMarco \cite[Theorem 1.2]{BD11} but we include the proof here for completeness and for the convenience of the reader. For  different proof (in the language of metrics on line bundle of $\mathbb{P}^1$) and ingredient we refer the reader to \cite[Corollary 14]{PST12}.\\
\begin{prop}\label{prop: shareheight} Let $K$ be a global field and let $f,g\in K(x)$ be rational functions of degree greater than $1$. Suppose that there exists an infinite distinct sequence $\{x_n\}_{n\in\mathbb{N}}\subset \mathbb{P}^1(\overline{K})$ such that 
$$\lim_{n\rightarrow\infty}\hat{h}_f(x_n)=\lim_{n\rightarrow\infty}\hat{h}_g(x_n)=0.$$ Then $\hat{h}_f=\hat{h}_g.$
\end{prop}
\begin{proof} Suppose that $\{x_n\}_{n\geq1}$ is a non-repeating sequence in $\mathbb{P}^1(\overline{K})$ such that $\hat{h}_f(x_n)\rightarrow0$ and $\hat{h}_g(x_n)\rightarrow0$ as $n\rightarrow+\infty$.\\
\noindent\textbf{Case I.} Suppose that $f,g\in \overline{\mathbb{Q}}(x)$. Suppose that $L$ is a number field for which both $f$ and $g$ are defined. Thus $\{x_n\}$ must be a sequence of algebraic points in $\mathbb{P}^1(\overline{L})$ for all $n\in \mathbb{N}$. Since $\{x_n\}$ is a sequence of point of small dynamical height with respect to both $f$ and $g$, so the sequence of probability measure supported equally on the sets  $\mathrm{Gal}(\overline{L}/L)\cdot x_n$  of conjugates of $x_n$ converges weakly to the canonical measures
		$\mu_{f,v}$ and $\mu_{g,v}$ at all places $v$ of $L$ in the Berkovich projective line $\mathbb{P}^1_{\mathrm{Berk},v}$ by the arithmetic equidistribution theorem of Baker-Rumely \cite{BR06}, Favre-Rivera-Letelier \cite{FRL06}, and Chambert-Loir \cite{CL06}. Thus $\mu_{f,v}=\mu_{g,v}$.
   From the definition of the Laplacian (\ref{eq:laplacian}),  for all $v\in M_L$ and all $x,y\in \mathbb{P}^1_{\mathrm{Berk},v}$, we have that the $v$-adic Arakelov Green's functions of $f$ and $g$ must differ by a constant $C_v(y)\in\mathbb{R}$  (\cite[Proposition 5.28]{BR10})\begin{equation}G_{\mu_f,v}(x,y)=G_{\mu_g,v}(x,y)+C_v(y)\label{eq: equivgreenconstantdependsony}.\end{equation} Using the fact that the Arakelov-Green's function is symmetric, fixing $x$ and viewing as a function of $y$, we can deduce that $C_v(y)= C_v\in\mathbb{R}$ is independent of $y$. Thus the equation (\ref{eq: equivgreenconstantdependsony}) becomes 
  \begin{equation}G_{\mu_f,v}(x,y)=G_{\mu_g,v}(x,y)+C_v.\label{eq: equivarakelovgreenfns}\end{equation}
  Summing equation (\ref{eq: equivarakelovgreenfns}) over all places $v\in M_L$ with appropriate local degree, it follows that
$$\hat{h}_f(x)+\hat{h}_f(y)=\hat{h}_g(x)+\hat{h}_g(y)+C$$ where $C=\sum_{v\in M_L}N_vC_v$ and is independent of $y$. 
Setting $y=x_n$ and letting $n\rightarrow\infty$, we have
$$\hat{h}_f(x)=\hat{h}_g(x)+C.$$ 
Taking $x$ to be $f$-preperiodic, we obtain $C=-\hat{h}_g(x)\leq0$. Similarly, taking $x$ to be $g$-preperiodic, it yields $C=\hat{h}_f(x)\geq0$. Hence $C = 0$ and the desired result follows immediately.\\
\noindent\textbf{Case II.} Let $L$ be a finitely generated field extension of $\overline{\mathbb{Q}}$ which is generated by all coefficients of $f$ and $g$. Suppose that neither $f$ nor $g$ is isotrivial (i.e., they are not conjugate to a rational function defined over $\overline{\mathbb{Q}}$).  Thus the sequence $\{x_n\}_{n\in\mathbb{N}}$ is defined over $\mathbb{P}^1(\overline{L}).$ Then the proof follows along similar lines to that of Case I. The last step to conclude that $C = 0$, we instead use the Theorem \ref{thm: BakBenIsotrivia} for function field.\\
\noindent\textbf{Case III.} Suppose that at least one of $f$ and $g$ is isotrivial. Without loss of generality, assume that $f$ is isotrivial (i.e., it is conjugate to a rational function defined over $\overline{\mathbb{Q}}$). Then we follow the argument used in Baker-DeMarco \cite[Theorem 2]{BD11} to deduce that $g$ must also be isotrivial. Hence both $f$ and $g$ are conjugate to a rational function defined over $\overline{\mathbb{Q}}$ and we are back to Case I. Therefore, the desired result is established.
\end{proof}

\noindent Recall that 
\begin{equation*}\mathrm{Prep}(f):=\{x\in\mathbb{C}\mid f^n(x)=f^m(x),\,\,\text{for}\,\,m,n\in\mathbb{N}\}\end{equation*}the collection of all preperiodic points of $f$. It is  infinite and countable \cite[$\S 6.2$]{Bea91}.
\begin{prop} \label{prop: degreegreaterthan1}
Let $f,g\in \mathbb{C}(x)$ be two rational functions  of degree greater than $1$. Suppose that the semigroup generated by $f$ and $g$ under composition is free. Then there are finitely many $\lambda\in \mathbb{C}$ such that there are $m,n\in\mathbb{N}$ with the following properties:
\begin{enumerate}
\item[(i)] $f^m(x)\neq c(x)$
\item[(ii)] $g^n(x)\neq c(x)$
\item[(iii)] $f^m(\lambda)=g^n(\lambda)=c(\lambda)$.
\end{enumerate}
\end{prop}
\begin{proof}  Let $K$ be a field generated by all coefficients of $f,g,$ and $c$  over $\mathbb{Q}$.  Thus $K$ is either a number field or a function field of finite transcendence degree with constant field $K\cap\overline{\mathbb{Q}}$.\\
\indent Suppose, on the contrary, that there is an infinite distinct sequence $\{\lambda_i\}_{i\in\mathbb{N}}$ such that for every $i$, there exist $m_i$ and $n_i$ so that $f^{m_i}(x)\neq c(x), g^{n_i}(x)\neq c(x)$, and $\lambda_i$ is solution to both $f^{m_i}(x)=c(x)$ and $g^{n_i}(x)=c(x)$. Note that $m_i,n_i\rightarrow+\infty$ as $i\rightarrow+\infty$. It follows from Lemma \ref{smallpoints} that 
$$\lim_{i\rightarrow\infty}\hat{h}_f(\lambda_i)=\lim_{i\rightarrow\infty}\hat{h}_g(\lambda_i)=0.$$
It is, then, immediately from Proposition \ref{prop: shareheight} that  $\hat{h}_f=\hat{h}_g$.\\
\noindent\textbf{Case I.} Suppose that $K$ is a number field. Proposition \ref{prep: heightpreper} asserts that the canonical heights $\hat{h}_f$ and $\hat{h}_g$ vanish precisely on the set of $f$-preperiodic points and $g$-preperiodic points, respectively. Then it follows that $\mathrm{Prep}(f)=\mathrm{Prep}(g)$ (i.e., $f$ and $g$ share the same set of preperiodic points). Applying the result of Bell-Haung-Peng-Tucker \cite[Corollary 4.12]{BHPT} to deduce that the semigroup generated by $f$ and $g$ under composition is not free. Hence it contradicts the assumption.\\
\noindent\textbf{Case II.} Suppose that $K$ is a function field of finite transcendence degree and neither $f$ nor $g$ is isotrivial. Thus $f$ and $g$ must share the same set of preperiodic points by Theorem \ref{thm: BakBenIsotrivia}. Again, applying the result of \cite{BHPT}, we have that $\langle f,g\rangle$ is not free.\\
\noindent\textbf{Case III.} Assume that at least one of $f$ and $g$ is isotrivial. Without loss of generality, assume that $f$ is isotrivial. We claim that $g$ must also be isotrivial. This follows immediately from a weak Northcott result of Baker \cite[Theorem 1.6]{Ba09}. Thus both $f$ and $g$ are conjugate to a rational function defined over $\overline{\mathbb{Q}}$. We are back to Case I. 
\end{proof}
Next, we treat the case when  one of rational functions $f$ and $g$ is of degree $1$. Here we briefly discuss the proof idea due to Hsia-Tucker \cite{HT17} and it goes through for rational functions $f$ and $g$ without any changes.
\begin{prop} \label{prop: degatleast1}
Let $f,g\in \mathbb{C}(x)$ be two rational functions  such that $\mathrm{deg}(f)>1$ and $\mathrm{deg}(g)=1$. Suppose that the semigroup generated by $f$ and $g$ under composition is free. Then there are finitely many $\lambda\in \mathbb{C}$ such that there are $m,n\in\mathbb{N}$ with the following properties:
\begin{enumerate}
\item[(i)] $f^m(x)\neq c(x)$
\item[(ii)] $g^n(x)\neq c(x)$
\item[(iii)] $f^m(\lambda)=g^n(\lambda)=c(\lambda)$.
\end{enumerate}
\end{prop}
\begin{proof}   Again, $K$ is a field generated by all coefficients of $f, g$ and $c$ over $\mathbb{Q}$. Thus $K$ is either a number field or a function field of finite transcendence degree.\\
\noindent\textbf{Case I.} Suppose that $K$ is a number field. We follow closely the argument of Hsia-Tucker  \cite[Proposition 9]{HT17}. For any infinite distinct sequence of $\{\lambda_i\}$ and $\{m_i\}$ such that $f^{m_i}(\lambda_i)=c(\lambda_i)$ for every $i\in\mathbb{N}$, we have that $\{\lambda_i\}$ is a sequence of small height with respect to $f$ by Lemma \ref{smallpoints}. Also, the condition (iii) yields that any solution $\lambda\in\mathbb{C}$ to the equation $g^n(\lambda)=c(\lambda)$ is of bounded degree (i.e., degree at most $\mathrm{deg}(c)$ since $\mathrm{deg}(g)=1$) over $K$. Hence the Northcott's property implies that there are only finitely many $\lambda\in\mathbb{C}$ satisfying (i), (ii), and (iii).\\
\noindent\textbf{Case II.} Suppose that $K$ is a function field of finite transcendence degree. We refer the reader to Hsia-Tucker \cite[Proposition 9]{HT17} as the proof is verbatim to the polynomial maps setting.
\end{proof}
We are now ready to prove the main result of this section. 
\begin{thm}
    Let $f(x)$ and $g(x)$ be rational functions in $\mathbb{C}(x)$, at least one of which has degree greater than one. Suppose that semigroup generated by $f$ and $g$ under composition is free and $c(x)$ is not a compositional power of $f$ or $g$. Then there are only finitely many $\lambda\in\mathbb{C}$ such that 
    $$f^m(\lambda)=g^n(\lambda)=c(\lambda).$$ for some $m,n\in\mathbb{N}$.
\end{thm}
\begin{proof} The proof divides into two cases:\\
\noindent\textbf{Case I.} If $\mathrm{deg} (f), \mathrm{deg} (g)>1$, then the desired result follows immediately from Proposition \ref{prop: degreegreaterthan1}. \\
\noindent\textbf{Case II.} If one of rational functions $f$ and $g$ is of degree one (without loss of generality, assume that $\mathrm{deg}(f)>1$ and $\mathrm{deg}(g)=1$), then the desired conclusion follows from Proposition \ref{prop: degatleast1}.
\end{proof}
	\section*{ Acknowledgements}
 We thank Jason Bell for helpful suggestions and conversations and Nathan Grieve for useful comments concerning the exposition. The second author was supported in part by NSERC grant RGPIN-2022-02951. We are grateful to the referee for their  careful reading and numerous helpful suggestions which greatly improve the article.


\begin{thebibliography}{[CNXZ]}
\bibitem[AR04]{AR04}
N.~Ailon~and~Z.~Rudnick.
\newblock \emph{Torsion points on curves and common divisors of {$a^k-1$} and {$b^k-1$}}. Acta Arith. \textbf{113} (2004), no. 1, 31--38.
\bibitem[Art06]{Art06}
E.~Artin.
\newblock \emph{Algebraic numbers and algebraic functions}.
\newblock (Reprint of the 1967 original). \newblock AMS Chelsea Publishing, Providence, RI, 2006.
\bibitem[BCZ03]{BCZ03}
Y.~Bugeaud,~P.~Corvaja,~and~U.~Zannier.
\newblock \emph{An upper bound for the {G}.{C}.{D}. of {$a^n-1$} and
{$b^n-1$}}. Math. Z. \textbf{243} (2003), no. 1,  79--84.
\bibitem[Ba09]{Ba09}
M.~Baker.
\newblock \emph{A finiteness theorem for canonical heights attached to rational maps over function fields}. J. Reine Angew. Math. \textbf{626} (2009),  205--233.
\bibitem[BD11]{BD11}
M.~Baker~and~L.~DeMarco.
\newblock \emph{Preperiodic points and unlikely intersections}. Duke Math. J. \textbf{159} (2011),  no. 1, 1--29.
\bibitem[BD13]{BD13}
M.~Baker~and~L.~DeMarco.
\newblock \emph{Special curves and postcritically finite polynomials}. Forum Math. Pi. \textbf{1} (2013),  e3, 35.
\bibitem[Bea91]{Bea91}
A.~F.~Beardon.
\newblock \emph{Iteration of rational functions}.
\newblock Graduate Texts in Mathematics, Vol. 132. 
Springer-Verlag, New York, 1991. 
\bibitem[Be05]{Be05}
R.~L.~Benedetto.
\newblock \emph{Heights and preperiodic points of polynomials over function fields}. Int. Math. Res. Not. (2005), no. 62, 3855--3866.
\bibitem[Be19]{Be19}
R.~L.~Benedetto.
\newblock \emph{Dynamics in one non-archimedean variable}. Graduate Studies in Mathematics, Vol.  198.
\newblock American Mathematical Society, Providence, RI, 2019.
\bibitem[BE87]{BE87}
I.~N.~Baker~and~A.~Er\"{e}menko.
\newblock \emph{A problem on {J}ulia sets}. Ann. Acad. Sci. Fenn. Ser. A I Math. \textbf{12} 
 (1987), no. 2, 229--236.
\bibitem[BG06]{BG06}
E.~Bombieri~and~Walter~Gubler.
\newblock \emph{Heights in {D}iophantine geometry}. New Mathematical Monographs, Vol. 4.
\newblock Cambridge University Press, Cambridge, 2006.
\bibitem[BGT]{BGT}
J.~Bell,~D.~Ghioca,~and~T.~Tucker.
\newblock \emph{The dynamical Mordell-Lang conjecture}. Mathematical Surveys and Monographs, Vol. 210.
\newblock American Mathematical Society, Providence, RI, 2016.
\bibitem[BHPT]{BHPT}
J.~Bell,~K.~Huang,~W.~Peng,~and~T.~Tucker.
\newblock \emph{A Tits alternative for endomorphisms of the projective line}. J. Eur. Math. Soc. {\bf 26} (2024), no.~12, 4903--4922.
\bibitem[BR06]{BR06}
M.~H.~Baker~and~R.~Rumely.
\newblock \emph{Equidistribution of small points, rational dynamics, and potential theory}. Ann. Inst. Fourier (Grenoble). \textbf{56} (2006), no. 3, 625--688.
\bibitem[BR10]{BR10}
M.~Baker~and~R.~Rumely.
\newblock \emph{Potential theory and dynamics on the {B}erkovich projective line}. Mathematical Surveys and Monographs, Vol.  159.
\newblock American Mathematical Society, Providence, RI, 2010.
\bibitem[CL06]{CL06}
A.~Chambert-Loir.
\newblock \emph{Mesures et \'{e}quidistribution sur les espaces de {B}erkovich}. J. Reine Angew. Math. \textbf{59} (2006), no. 3, 215--235.
\bibitem[CS93]{CS93}
G.~S.~Call~and~J.~H.~Silverman.
\newblock \emph{Canonical heights on varieties with morphisms}. Compositio Math. \textbf{89} (1993), no. 2,  163--205.
\bibitem[CZ13]{CZ13}
P.~Corvaja~and~U.~Zannier.
\newblock \emph{Greatest common divisors of {$u-1,\ v-1$} in positive characteristic and rational points on curves over finite fields}. J. Eur. Math. Soc. (JEMS). \textbf{15} (2013), no. 5,  1927--1942.
\bibitem[EG15]{EG15}
J.-H. Evertse and K. Gy\H{o}ry,
\emph{Unit equations in Diophantine number theory}.
Cambridge Studies in Advanced Mathematics, Vol. 146.
\newblock Cambridge University Press, Cambridge, 2015.
\bibitem[FRL06]{FRL06}
C.~Favre~and~J.~Rivera-Letelier.
\newblock \emph{\'{E}quidistribution quantitative des points de petite hauteur sur la droite projective}. Math. Ann. \textbf{335} (2006), no. 2, 311--361.
\bibitem[GHT13]{GHT13}
D.~Ghioca,~L.-C.~Hsia~and~ T.~J.~Tucker.
\newblock \emph{Preperiodic points for families of polynomials}. Algebra Number Theory. \textbf{7} (2013), no. 3, 701--732.
\bibitem[GHT15]{GHT15}
D.~Ghioca,~L.-C.~Hsia~and~ T.~J.~Tucker.
\newblock \emph{Preperiodic points for families of rational maps}. Proc. Lond. Math. Soc. (3). \textbf{110} (2015), no. 2,  395--427.
\bibitem[GHT17]{GHT17}
D.~Ghioca,~L.-C.~Hsia~and~ T.~J.~Tucker.
\newblock \emph{On a variant of the {A}ilon-{R}udnick theorem in finite characteristic}. New York J. Math. \textbf{23} (2017),  213--225.
\bibitem[GHT18]{GHT18}
D.~Ghioca,~L.-C.~Hsia~and~ T.~J.~Tucker.
\newblock \emph{A variant of a theorem by {A}ilon-{R}udnick for elliptic curves}. Pacific J. Math. \textbf{295} (2018), no. 1,  1--15.
\bibitem[Gr20]{Gr20}
N.~Grieve.
\newblock \emph{Generalized {GCD} for toric {F}ano varieties}. Acta Arith. \textbf{195} (2020), no. 4, 415--428.
\bibitem[GW20]{GW20}
N.~Grieve~and~J.~T-Y.~Wang.
\newblock \emph{Greatest common divisors with moving targets and consequences
              for linear recurrence sequences}. Trans. Amer. Math. Soc. \textbf{373} (2020), no. 2, 8095--8126.
\bibitem[HT17]{HT17}
L.-C.~Hsia~and~ T.~J.~Tucker.
\newblock \emph{Greatest common divisors of iterates of polynomials}. Algebra and Number Theory. \textbf{11} (2017), no. 6, 1437--1459.
\bibitem[Hu20]{Hu20}
K.~Huang.
\newblock \emph{Generalized greatest common divisors for orbits under rational functions}. Monatsh. Math. \textbf{191} (2020), no. 1,  103--123.
\bibitem[KT21]{KT21}
A.~Kolpakov~and~A.~Talambutsa.
\newblock \emph{On free semigroups of affine maps on the real line}. Proc. Amer. Math. Soc. \textbf{150} (2021), no. 6, 2301--2307.
\bibitem[La83]{La83}
S.~Lang.
\newblock \emph{Fundamentals of {D}iophantine geometry}. 
\newblock Springer-Verlag, New York, 1983.
\bibitem[Le19]{Le19}
A.~Levin.
\newblock \emph{Greatest common divisors and {V}ojta's conjecture for blowups of algebraic tori}. Invent. Math. \textbf{215} (2019), no. 2, 493--533.
\bibitem[Lu05]{Lu05}
F.~Luca.
\newblock \emph{On the greatest common divisor of {$u-1$} and {$v-1$} with {$u$} and {$v$} near {$S$}-units}. Monatsh. Math. \textbf{146} (2005), no. 3, 239--256.
\bibitem[Os16]{Os16}
A.~Ostafe.
\newblock \emph{On some extensions of the {A}ilon-{R}udnick theorem}. Monatsh. Math. \textbf{181} (2016), no. 2, 451--471.
\bibitem[PST09]{PST09}
C.~Petsche,~L.~Szpiro~and~ M.~Tepper.
\newblock \emph{Isotriviality is equivalent to potential good reduction for endomorphisms of {$\Bbb P^N$} over function fields}. J. Algebra. \textbf{322} (2009), no. 9, 3345--3365.
\bibitem[PST12]{PST12}
C.~Petsche,~L.~Szpiro~and~ T.~J.~Tucker.
\newblock \emph{A dynamical pairing between two rational maps}. Trans. Amer. Math. Soc. \textbf{364} (2012), no. 4, 1687--1710.
\bibitem[Si04a]{Si04a}
J.~H.~Silverman.
\newblock \emph{Common divisors of {$a^n-1$} and {$b^n-1$} over function fields}. New York J. Math. \textbf{10} (2004),  37--43.
\bibitem[Si04b]{Si04b}
J.~H.~Silverman.
\newblock \emph{Common divisors of elliptic divisibility sequences over function fields}. Manuscripta Math. \textbf{114} (2004), no. 4, 431--446.
\bibitem[Si05]{Si05}
J.~H.~Silverman.
\newblock \emph{Generalized greatest common divisors, divisibility sequences, and {V}ojta's conjecture for blowups}. Monatsh. Math. \textbf{145} (2005), no. 4, 333--350.
\bibitem[Si07]{Si07}
J.~H.~Silverman.
\newblock \emph{The arithmetic of dynamical systems}. Graduate Texts in Mathematics, Vol.  241.
\newblock Springer, New York, 2007.
\bibitem[SS95]{SS95}
W.~Schmidt~and~N.~Steinmetz.
\newblock \emph{The polynomials associated with a {J}ulia set}. Bull. London Math. Soc. \textbf{27} 
 (1995), no. 3, 239--241.
\bibitem[Za12]{Za12}
U.~Zannier.
\newblock \emph{Some problems of unlikely intersections in arithmetic and geometry}. (With appendixes by David Masser). Annals of Mathematics Studies, Vol.  181.
\newblock Princeton University Press, Princeton, NJ, 2012.
	\end{thebibliography}
\end{document}